\documentclass{amsart}

\usepackage[usenames]{color} 
\usepackage{amssymb} 
\usepackage{amsmath} 
\usepackage{amsthm}
\usepackage[utf8]{inputenc} 

\usepackage[colorlinks, linkcolor=blue,  citecolor=blue, urlcolor=blue]{hyperref}%

\usepackage{url}
\usepackage{amssymb}
\usepackage{verbatim}
\usepackage{alltt}
\usepackage{subfigure}
\usepackage{graphicx}
\usepackage{algorithm}
\usepackage{algorithmic}
\usepackage{float}
\usepackage{enumerate}
\usepackage[scientific-notation=true]{siunitx}

\graphicspath{{figures/}}

\newtheorem{theorem}{Theorem}[section]

\newtheorem{lemma}[theorem]{Lemma}
\newtheorem{proposition}[theorem]{Proposition}
\newtheorem{definition}[theorem]{Definition}
\theoremstyle{remark}
\newtheorem{remark}{Remark}[section]
\theoremstyle{definition}
\newtheorem{example}{Example}[section]

\newcommand{\Z}{\mathbb{Z}}
\newcommand{\R}{\mathbb{R}}

\newcommand{\Rn}{\R^n}

\newcommand{\bO}{\mathcal{O}}

\newcommand{\bq}{\begin{equation}}
\newcommand{\eq}{\end{equation}}

\newcommand{\boundary}{\partial \Omega}

\newcommand{\norm}[1]{\Vert#1\Vert}

\newcommand{\abs}[1]{\vert#1\vert}

\usepackage{xspace}

\DeclareMathOperator{\grad}{\nabla}
\DeclareMathOperator{\argmin}{{argmin}}
\DeclareMathOperator{\argmax}{{argmax}}

\newcommand{\widthtwofigures}{0.48\textwidth}
\newcommand{\widththreefigures}{0.31\textwidth}

\usepackage{caption}

\newcommand{\xs}{x^*}
\newcommand{\vs}{v^*}

\begin{document}

\title[star-shaped envelope]
{A Partial Differential Equation Obstacle Problem for the Level Set Approach to Visibility}

\author{Adam M. Oberman}
\address{Department of Mathematics and Statistics, McGill University, 805 Sherbrooke Street West,
Montreal, Quebec, H3A 0G4, Canada ({\tt adam.oberman@mcgill.ca}).}

\author{Tiago Salvador}
\address{Department of Mathematics, University of Michigan, 530 Church St. Ann Arbor, MI 48105 ({\tt saldanha@umich.edu}).}

\date{\today}

\keywords{Visibility, Level set method, Viscosity solutions, Fast sweeping method}


\begin{abstract}
In this article we consider the problem of finding the visibility set from a given point when the obstacles are represented as the level set of a given function.
Although the visibility set can be computed efficiently by ray tracing, there are advantages to using a level set representation for the obstacles, and to characterizing the solution using a Partial Differential Equation (PDE).
A nonlocal PDE formulation was proposed in Tsai et. al. (Journal of Computational Physics 199(1):260-290, 2004) \cite{TsaiVisibilityPDE}: in this article we propose a simpler PDE formulation, involving a nonlinear  obstacle problem.  We present a simple numerical scheme and show its convergence using the framework of Barles and Souganidis. Numerical examples in both two and three dimensions are presented.
\end{abstract}

\maketitle

%

\section{Introduction}

In this article we consider the problem of finding the visibility set from a given viewpoint given a set of known obstacles using a Partial Differential Equation (PDE).
In principle, the visibility set is simply given by ray tracing and there are numerous algorithms for solving the visibility problem using explicit representations of the obstacles \cite{Coorg1997,Durand2000,AgarwalRayShooting2D,AgarwalRayShooting3D}.

Finding the visibility set plays a crucial role in numerous  applications including rendering, visualization \cite{VisualizationReference}, etching \cite{EtchingReference}, surveillance, exploration \cite{TsaiInformationVisibility}, navigation \cite{LandaTsaiVisibilityPointClouds}, and inverse problems, to only name a few.
Specifically, in \cite{TsaiVisibilityExtensions} the level set framework \cite{OSnum,SethianBook} developed in \cite{TsaiVisibilityPDE} was extended to deal with the optimal placing of a single viewer or a group of viewers and $A$-to-$B$ optimal path planning, where optimality is measured in terms of the volume of the visible region.
More recently, in \cite{TsaiDeepLearning} a convolutional neural network is proposed to determine the vantage points that maximize visibility in the context of surveillance and exploration, with the visibility sets of the training data being computed efficiently using the PDE formulation introduced in \cite{TsaiVisibilityPDE}.
For applications which involve optimization of the viewpoint, the discontinuity of the visibility can make optimization more difficult.
The advantage of using level set/PDE methods is the improved regularity of the solution.

It is clear then that the ultimate goal of the work is inverse problems involving visibility.
As is the case with inverse problems, a better understanding of the forward problem is essential for better results of the more challenging inverse problem.
In this work, we focus our attention in the forward problem and do not go further and study the inverse problem.
We propose a simple formulation of the visibility problem - the visibility set is the subzero level set of the solution of a nonlinear obstacle problem.

In \cite{TsaiVisibilityPDE} the visibility problem is presented as a boundary value problem for a first order differential equation: the visibility set to a given viewpoint $\xs$ is given by $\{\psi(x) \geq 0\}$ where the function $\psi(x)$ is the solution of
\bq\label{PDE:Tsai}
\grad \psi \cdot \frac{x-\xs}{\abs{x-\xs}} = \min\left\{H(\psi-g)\grad g \cdot \frac{x-\xs}{\abs{x-\xs}},0\right\}
\eq
with $\psi(\xs) = g(\xs)$.
Here $H(z) = \chi_{[0,\infty)}(z)$ is the characteristic function of $[0,\infty)$ and $g$ is a signed distance function to the obstacles, positive outside the obstacles and negative inside.
Despite the complex nature of the operator in \eqref{PDE:Tsai}, in \cite{TsaiPDEProof} the visibility function $\psi$ is shown to be the viscosity solution of an equivalent Hamilton-Jacobi type equation involving jump discontinuities in the Hamiltonian.
A numerical scheme to solve this equation is presented and its convergence is established.


For our formulation, $g$ is still a signed distance function to the obstacles, but is instead negative outside and positive inside. Then the visibility set is given as $\{u(x) \leq 0\}$ where the function $u(x)$ solves the following nonlinear first order local PDE
\[
\min\{u(x)-g(x), (x-\xs)\cdot \grad u(x)\} = 0
\]
with $u(\xs) = g(\xs)$. This is not only considerably simpler than \eqref{PDE:Tsai}, but can also be generalized to allow multiple viewpoints as we will show. Moreover, each sublevel set of $u$ is in fact the visibility set of the corresponding superlevel set of $g$. Efficiencies then arise when the obstacles are given by the graph of a function (for example, heights of buildings). In this case, we can reduce the dimension of the problem, and compute the horizontal visibility set from a given height, using the level set representation. Similarly, if the $t$ superlevel set of $g$ represents the position of the obstacles at a certain time $t$ ($g$ can for instance be the solution of an Eikonal equation), then the PDE needs to be solved only once and the visibility set at any given time can be extracted from the corresponding sublevel set.

The paper is organized as follows.
In \autoref{sec:SS}, we characterize visibility sets as star-shaped envelopes.
In \autoref{sec:PDE} we derive the new visibility PDE and its generalization to multiple viewpoints.
In \autoref{sec:scheme} we present the numerical scheme, while in \autoref{sec:convergence} we establish its convergence.
Finally, in \autoref{sec:numerics} we present both two-dimensional and three-dimensional examples of visibility sets computed using the new proposed PDE.

\section{Star shaped sets and functions}\label{sec:SS}

In this section we give an interpretation of the visibility set from a given point $\xs$ as the star-shaped envelope with respect to the point $\xs$. The definitions of star-shaped sets and envelopes are then extended to functions. Finally, we provide explicit formulas for the star-shaped envelopes of a function.

We star by recalling the definition of a star-shaped set.

\begin{definition}
We say the set $S \subset \Rn$ is \emph{star-shaped} with respect to $\xs$ if 
\[
x \in S \implies t\xs + (1-t)x \in S, \quad \text{ for all } t \in [0,1].
\]
\end{definition}

A simple example of a star-shaped set is a convex set. Indeed, convex sets are star-shaped with respect to every point inside. Moreover, intersections of convex sets are convex, but unions are not. As a consequence there is a natural (outer) convex envelope, but not an inner one. On the other hand, star-shaped sets are closed under both intersections and unions, which means one can define two star-shaped envelopes (with respect to $x^*$) for sets, the inner and outer envelopes.

\begin{definition}
Given $S\subset \R^n$ and $\xs \in S$, the outer star-shaped envelope  of $S$ with respect to $\xs$ is the intersection of all star-shaped sets with respect to $\xs$ which contain $S$. The inner star-shaped envelope of $S$ is the union of all star-shaped sets contained in it.
\end{definition}

\begin{figure}[h]
\centering
\subfigure{\includegraphics[width=\widththreefigures,trim={4cm 2.3cm 5cm 2.3cm},clip]{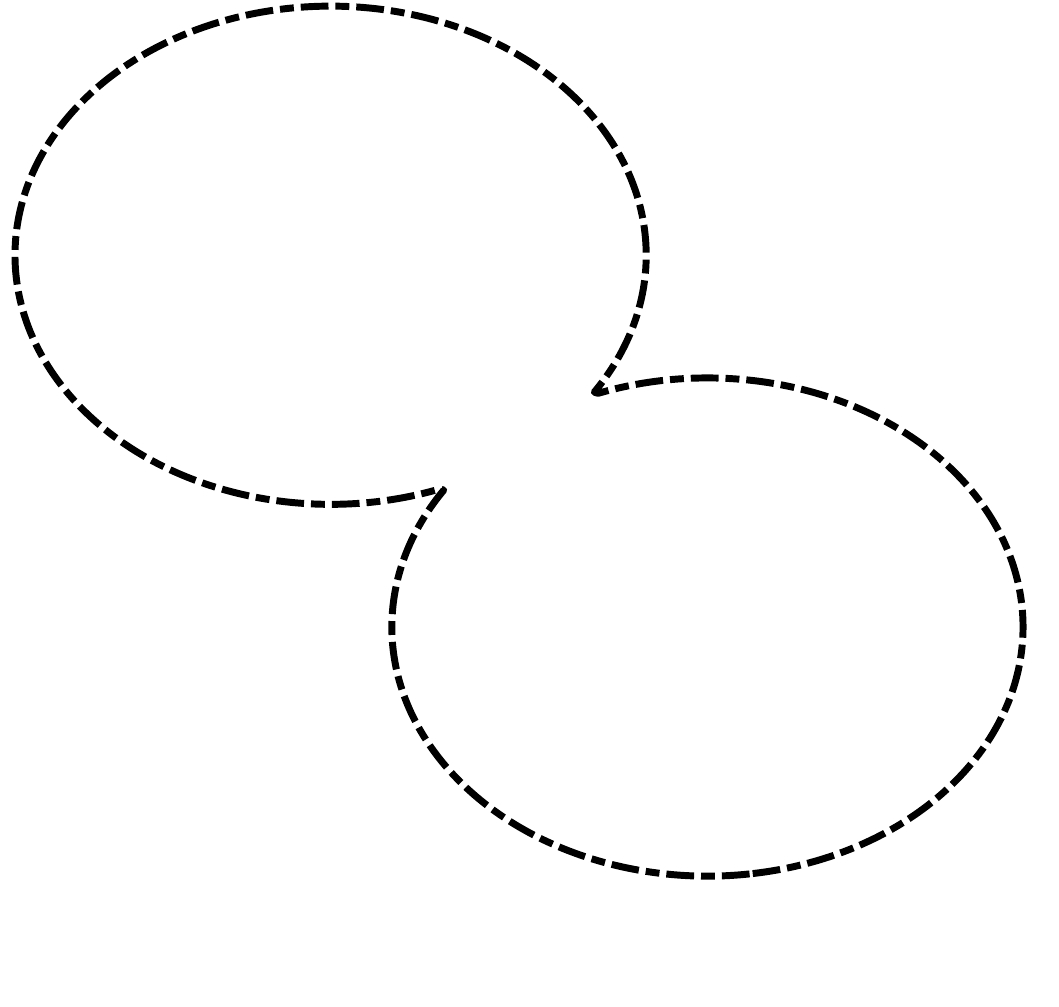}}
\subfigure{\includegraphics[width=\widththreefigures,trim={4cm 2.3cm 5cm 2.3cm},clip]{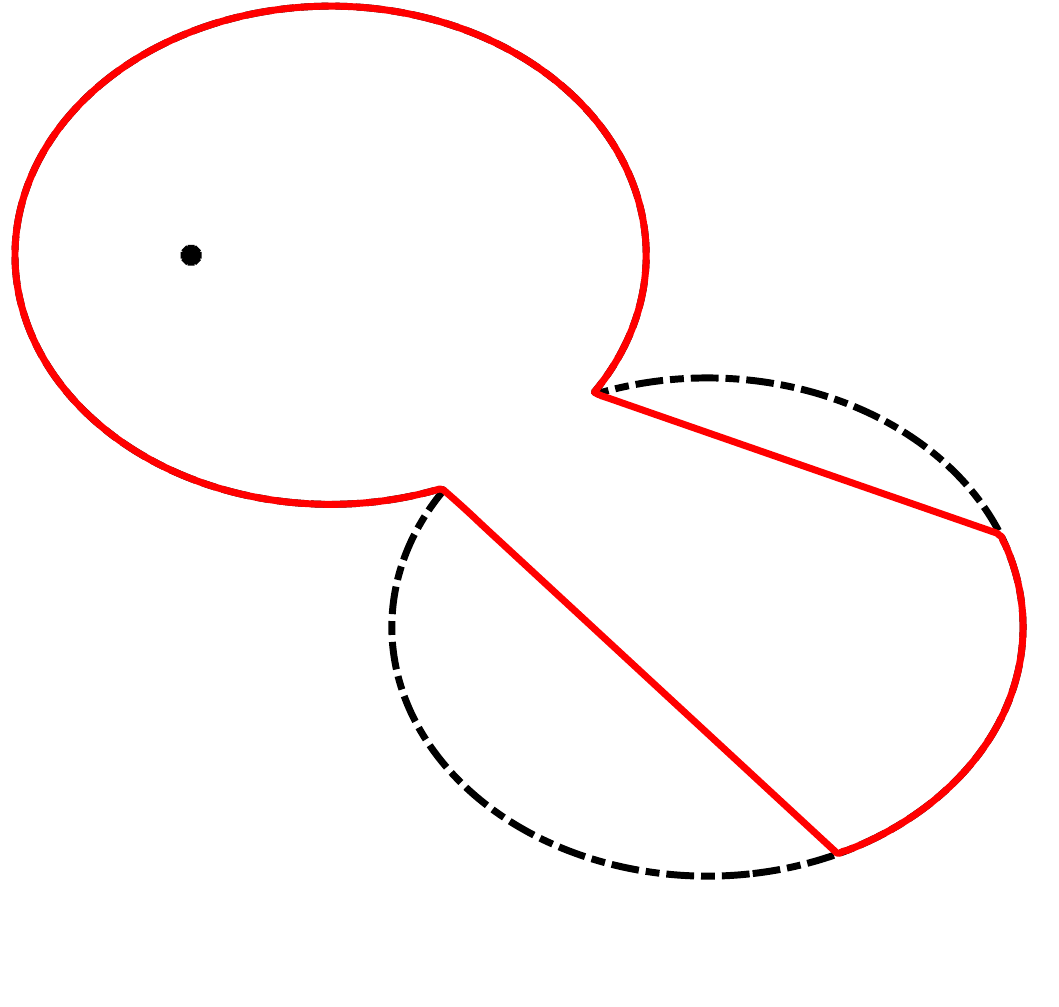}}
\subfigure{\includegraphics[width=\widththreefigures,trim={4cm 2.3cm 5cm 2.3cm},clip]{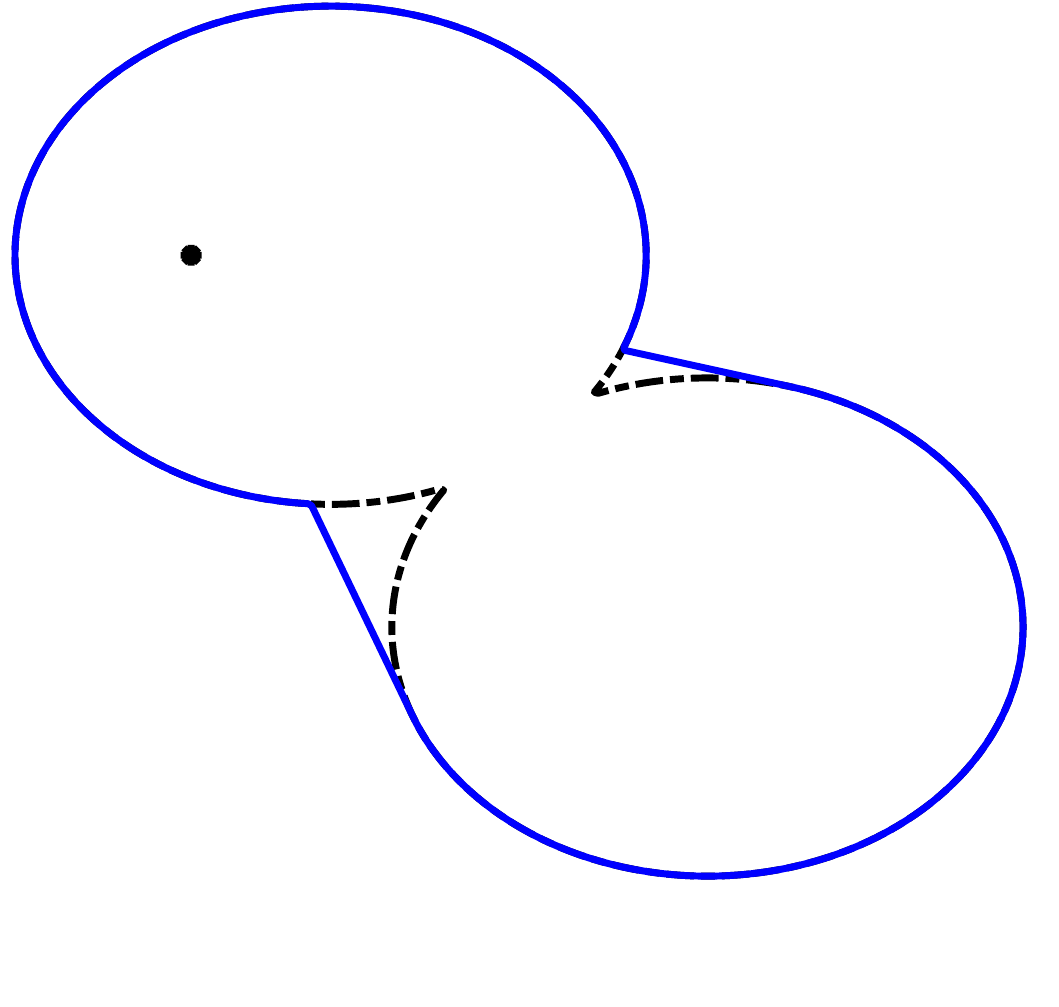}}
\caption{The inner and outer star-shaped envelopes of a set: the original set (left, dashed); inner star-shaped envelope / visibility set (center); outer star-shaped envelope (right).}
\label{fig:SSenvelopes}
\end{figure}
	
Looking at Figure \ref{fig:SSenvelopes}, one immediately sees how the inner star-shaped envelope corresponds to the \emph{visibility} subset of $S$ if there is an illumination source at the point $\xs$.
The remainder of $S$ is the invisible part.
One the other hand, the outer star-shaped envelope minus $S$ corresponds to the least amount of obstacles which would need to be removed so that all of $S$ is visible from $\xs$.

We now discuss star-shaped functions which are the main building block to characterize the visibility set as the solution of a nonlinear obstacle PDE. We write  $S_\alpha(u)\equiv\{x \in \Rn \mid u(x)\leq\alpha\}$ for the $\alpha$-sublevel set of a function $u$ and let $\Omega$ be a star-shaped domain with respect to $\xs$.

\begin{definition}
We say that the function $u: \Omega \to \R$ is star-shaped with respect to $\xs$ if
\[
\text{$S_\alpha(u)$ is star-shaped with respect to $\xs$}
\]
for all $\alpha\in\R$.
\end{definition}

\begin{remark}
Notice the similarity to quasiconvex functions: $u$ is said to be quasiconvex if $S_\alpha(u)$ is convex for all $\alpha\in\R$.
\end{remark}

We now characterize star-shaped functions with a zero-order condition.

\begin{lemma}\label{lemma:SSequivalence}
A function $u:\Omega \to \R$ is star-shaped with respect to $\xs$ if and only if  
\begin{equation}
\label{eq:SScond0}
u(t \xs + (1-t)y) \le  u(y), 
\quad \text{ for all } y \in \Omega,  0 \le t \le 1. 	
\end{equation}
\end{lemma}

\begin{proof}
By definition, $u$ is star-shaped with respect to $\xs$ if and only if for all $\alpha\in\R$, $S_\alpha(u)$ is star-shaped with respect to $\xs$. This is equivalent to the condition 
\[
u(y) \le \alpha \implies u(t \xs + (1-t)y) \le \alpha,
\quad \text{ for all } y \in \Omega,  0 \le t \le 1
\]
for all all $\alpha \in \R$, which in turn is equivalent to \eqref{eq:SScond0}.
\end{proof}

We use this result to describe the monotonicity of a star-shaped function.

\begin{proposition}\label{prop:SSequivalence}
Let $u:\Omega\to\R$ be a function. Then $u$ is star-shaped with respect to $\xs$ if and only if $u$ is increasing along rays from $\xs$ to $x$. Moreover, if $u$ is star-shaped with respect to $\xs$, then $\xs$ is a global minimum of $u$.
\end{proposition}
\begin{proof}
This follows immediately from Lemma \ref{lemma:SSequivalence}.
\end{proof}

Next, in a similar way to star-shaped envelopes of a set, we define upper and lower star-shaped envelopes of a function $g$ with respect to a point $\xs$.

\begin{definition}
Let $g\in C(\Omega)$ be bounded by below. The lower star-shaped envelope of $g$ with respect to $\xs$ is defined as 
\begin{equation}
SS^-(g)(x) = \sup \{ v(x) \mid \text{ $v$ is star-shaped with respect to $\xs$ and $v \leq g$}\},
\end{equation}
while the upper star-shaped envelope is given by
\begin{equation}
SS^+(g)(x) = \inf  \{ v(x) \mid \text{ $v$ is star-shaped with respect to $\xs$ and $v \geq g$}\}.
\end{equation}
\end{definition}

\begin{remark}
We require that $g$ is bounded by below in order for the lower star-shaped envelope to be well defined since otherwise there would no star-shaped function with respect to $\xs$ bounded from above by $g$.
\end{remark}

We finish this section by proving the following simple explicit solution formulas for the star-shaped envelope of a function.

\begin{proposition}\label{prop:LowerEnvExplicit}
Let $g \in C(\Omega)$ be bounded by below and define $w:\Omega\to\R$ to be given by
\bq\label{eq:lowerss_formula} 
w(x) = \min \left\{g(y) \mid y = x + t(x-\xs) \in \Omega,  t\geq 0\right\}
\eq
with $w(\xs) = \min_{x\in\Omega} g(x)$. Then $w$ is star-shaped with respect to $\xs$ and $w = SS^-(g)$.
\end{proposition}
\begin{proof}
By the assumptions on $g$, $w$ is well-defined. By construction, $w$ is increasing along rays from $\xs$ to $x$, and so, by Proposition \ref{prop:SSequivalence}, $w$ is star-shaped with respect to $\xs$. Moreover, it is clear that $w \leq g$.

We want to show that $w = SS^-(g)$. Suppose by contradiction that it is not. This means that there exists a star-shaped with respect to $\xs$ function $v$ with $v\leq g$ such that $v(x) > w(x)$ for some $x\in\Omega$. Without loss of generality, assume that $x\neq\xs$. We have $w(x) < g(x)$ and that there exists $y \in \Omega$ such that
\[
y \in \argmin \left\{g(y) \mid y = x+t(x-\xs)\in\Omega,t > 0\right\}.
\]
Hence $v(x) > w(x) = w(y) = g(y) \geq v(y)$ and therefore $v$ is not increasing along the ray from $\xs$ to $x$. Finally, we invoke Proposition \ref{prop:SSequivalence} to conclude that $v$ is not star-shaped with respect to $\xs$, which leads to the desired contradiction.
\end{proof}

\begin{remark}\label{rmk:prop_lower_ss}
Intuitively, we can find $w$ by tracing the values from the boundary along rays to $\xs$ and taking the minimum of $g$ along the way.
\end{remark}

\begin{proposition}\label{prop:UpperEnvExplicit}
Let $g \in C(\Omega)$ and define $u:\Omega\to\R$ to be given by
\bq\label{eq:explicit_visibility}
u(x) = \max \left\{g(y) \mid y = \xs + t(x-\xs) \in \Omega, t\in[0,1]\right\}.
\eq
Then $u$ is star-shaped with respect to $\xs$ and $u = SS^+(g)$.
\end{proposition}
\begin{proof}
From the definition of $u$, it is clear that $u$ is increasing along rays from $\xs$ to $x$, and so, by Proposition \ref{prop:SSequivalence}, $u$ is star-shaped with respect to $\xs$. Moreover, $u \geq g$, again by definition of $u$.

We want to show that $u = SS^+(g)$. Suppose by contradiction that it is not. This means that there exists a star-shaped with respect to $\xs$ function $v$ with $v\geq g$ such that $v(x) < u(x)$ for some $x\in\Omega$. Without loss of generality assume $x\neq\xs$. We have $u(x) > g(x)$ and that there exists $y\in\Omega$ such that
\[
y \in \argmax \left\{g(y) \mid y = \xs+t(x-\xs)\in\Omega,t\in[0,1)\right\}.
\]
Hence $v(x) < u(x) = u(y) = g(y) \leq v(y)$, which means, just like in the proof of Proposition \ref{prop:LowerEnvExplicit}, that $v$ is not increasing along the ray from $\xs$ to $x$. Hence $v$ is not star-shaped with respect to $\xs$ according to Proposition \ref{prop:SSequivalence} and we have obtained our contradiction.
\end{proof}
\begin{remark}\label{rmk:prop_upper_ss}
This formula corresponds to the classic ray tracing algorithm to find the visibility set. We trace the values towards the boundary along rays from $\xs$ taking the maximum of $g$ along the way.
\end{remark}

\section{PDEs and Visibility}\label{sec:PDE}

In this section, we present the new PDE formulation of visibility sets from a single viewpoint and its extension to multiple viewpoints.
We start with a level set PDE interpretation for star-shaped envelopes.
Given that the inner star-shaped envelope of a set corresponds to its visibility set, the PDE obtained computes the visibility set for each sublevel set. We then generalize it to multiple viewpoints.

\subsection{Viscosity Solutions}
Viscosity solutions \cite{CIL} provide the correct notion of weak solution to a class of degenerate elliptic PDEs which includes the PDEs considered here.
We review it briefly here.

Let $S^n$ be the set of real symmetric $n\times n$ matrices, and take $N \leq M$ to denote the usual partial ordering on $S^n$, namely that $N-M$ is negative semi-definite.
\begin{definition}
The operator $F(x,r,p,M):\Omega\times \R\times \R \times S^n \to \R$ is degenerate elliptic if
\[
F(x, r, p, M) \leq F(x, s, p, N) \quad \text{whenever } r \leq s \text{ and } N \leq M.
\]

\begin{remark}
For brevity we use the notation $F[u](x) \equiv F(x, u(x), \grad u(x), D^2u(x))$.
\end{remark}
\end{definition}


\begin{definition}[Upper and lower semi-continuous envelopes]
The upper and lower semicontinuous envelopes of a function $u(x)$ are defined, respectively, by
\begin{align*}
u^*(x) & = \limsup_{y\to x} u(y),\\
u_*(x) & = \liminf_{y\to x} u(y).
\end{align*}
\end{definition}

\begin{definition}[Viscosity solutions]
Let $F:\Omega \times \R \times \Rn$. We say the upper semi-continuous (lower semi-continuous) function $u:\Omega\to\R$ is a viscosity subsolution (supersolution) of $F[u] = 0$ in $\Omega$ if for every $\phi \in C^1(\Omega)$, whenever $u-\phi$ has a local maximum (minimum) at $x\in\Omega$,
\[
F(x, u(x), \grad \phi(x)) \leq 0 \ (\geq 0).
\]
Moreover, we say u is a viscosity solution of $F[u] = 0$ if $u$ is both a viscosity sub- and supersolution.
\end{definition}

\begin{remark}
For brevity we use the notation $F[u](x) \equiv F(x,u(x),\grad u(x))$. In addition, when checking the definition of a viscosity solution we can limit ourselves to considering unique, strict, global maxima (minima) of $u-\phi$ with a value of zero at the extremum. See, for example, \cite[Prop 2.2]{KoikeViscosity}.
\end{remark}

\subsection{Regularity}

We briefly discuss the regularity of the star-shaped envelopes. We start by observing that the star-shaped functions need not be continuous.

\begin{example}
In one dimension, the function $u(x)=1$ for $x\neq 0$ and $u(0)=0$ is star-shaped with respect to $0$.
In two dimensions,  take $A = \{(x,y)\in\R^2: xy = 0\}$ and define $u$ as the characteristic function of the complement of $A$, i.e., $u(x) = 0$ if $x \in A$ and $u(x) = 1$ otherwise.
Once again $u$ is star-shaped with respect to the origin, but it is not continuous.
In fact, $u$ is lower semicontinuous.
\end{example}

As for the star-shaped envelopes of functions, the upper star-shaped envelope is continuous while the lower star-shaped envelope is only lower semicontinuous as it may be discontinuous at $\xs$.

\begin{proposition}\label{prop:continuity}
Let $g \in C(\Omega)$ be bounded by below. Then $w = SS^-(g)$ is lower semicontinuous in $\Omega$ and continuous in $\Omega$ except at $\xs$, while $u = SS^+(g)$ is continuous in $\Omega$. If $\xs$ is a global minimum of $g$ then $w = SS^-$ is also continuous in $\Omega$.
\end{proposition}

\begin{proof}
The proof follows from the solutions formulas \eqref{eq:lowerss_formula} and \eqref{eq:explicit_visibility} since we take the minimum and maximum of a continuous function $g$ along rays to and from $\xs$, respectively, as pointed out in Remarks \ref{rmk:prop_lower_ss} and \ref{rmk:prop_upper_ss}.
\end{proof}

\begin{example}\label{ex:regularity}
Let $g(x) = \abs{x+1}$ and let $\xs = 0$. Then the lower and upper star-shaped envelopes are given by
\[
w(x) = \begin{cases}
-x-1	& \text{if } x < -1,\\
0		& \text{if } -1 \leq x \leq 0,\\
x+1		& \text{if } x>0,
\end{cases}
\quad \text{and} \quad u(x) = \begin{cases}
-x-1	& \text{if } x < -2,\\
1		& \text{if } -2 \leq x \leq 0,\\
x+1		& \text{if } x>0.
\end{cases}
\]

A two-dimensional example is given in  Figure~\ref{fig:Ex1Envelope}: $g$ is given by the distance to two points and $\xs$ is chosen as a point on the $.3$ level set of $g$.
The upper and lower envelopes are pictured. 
The lower one is discontinuous at $\xs$.
Replacing $g$ with $\max(g,g(\xs)$ leads to a function whose global minimum is attained at $\xs$ and therefore both star-shaped envelopes are continuous in this case. This is depicted in Figure~\ref{fig:Ex1Envelope2}.
\end{example}

\begin{figure}[h]
\subfigure{\includegraphics[width=\widththreefigures]{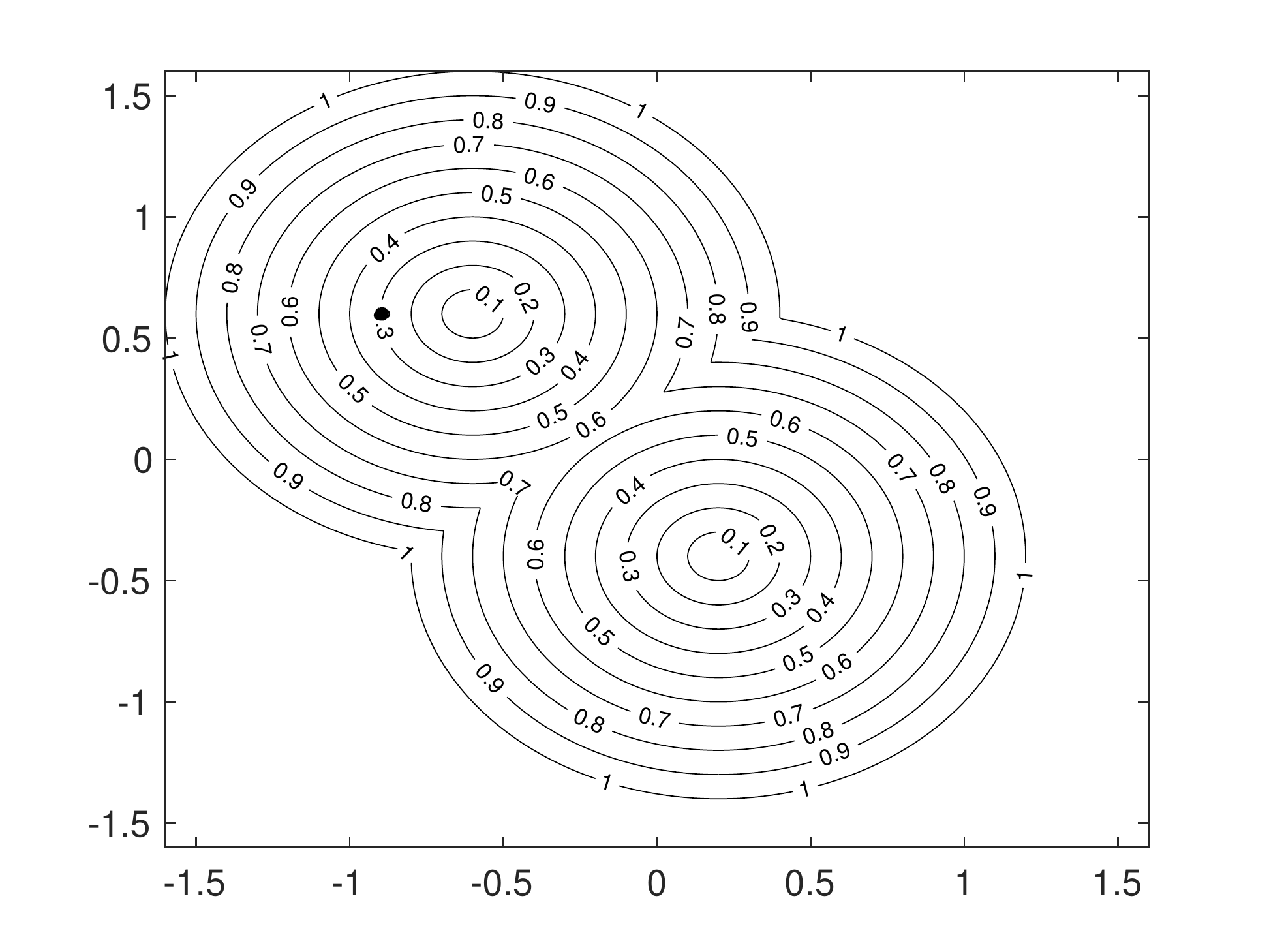}}
\subfigure{\includegraphics[width=\widththreefigures]{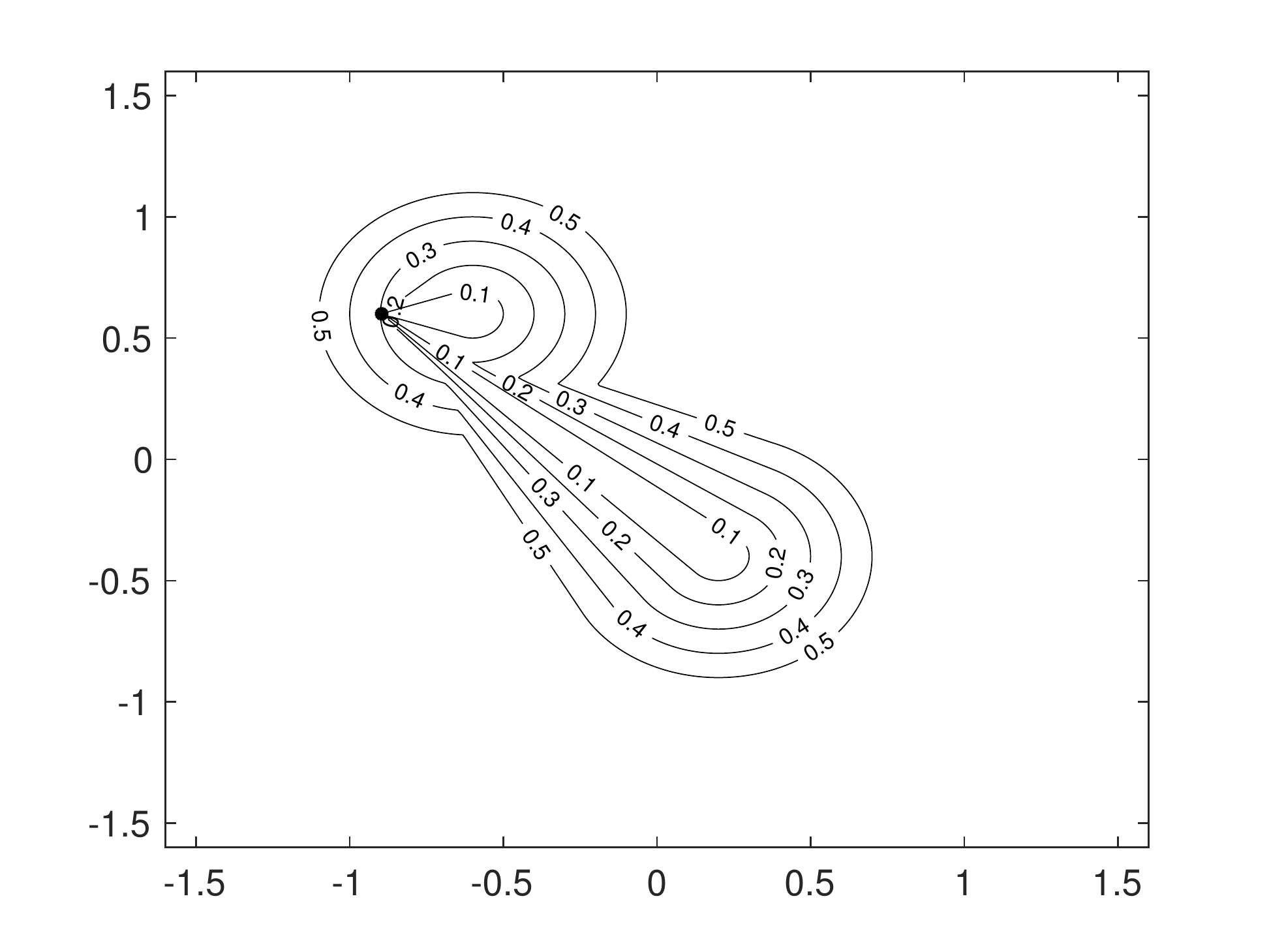}}
\subfigure{\includegraphics[width=\widththreefigures]{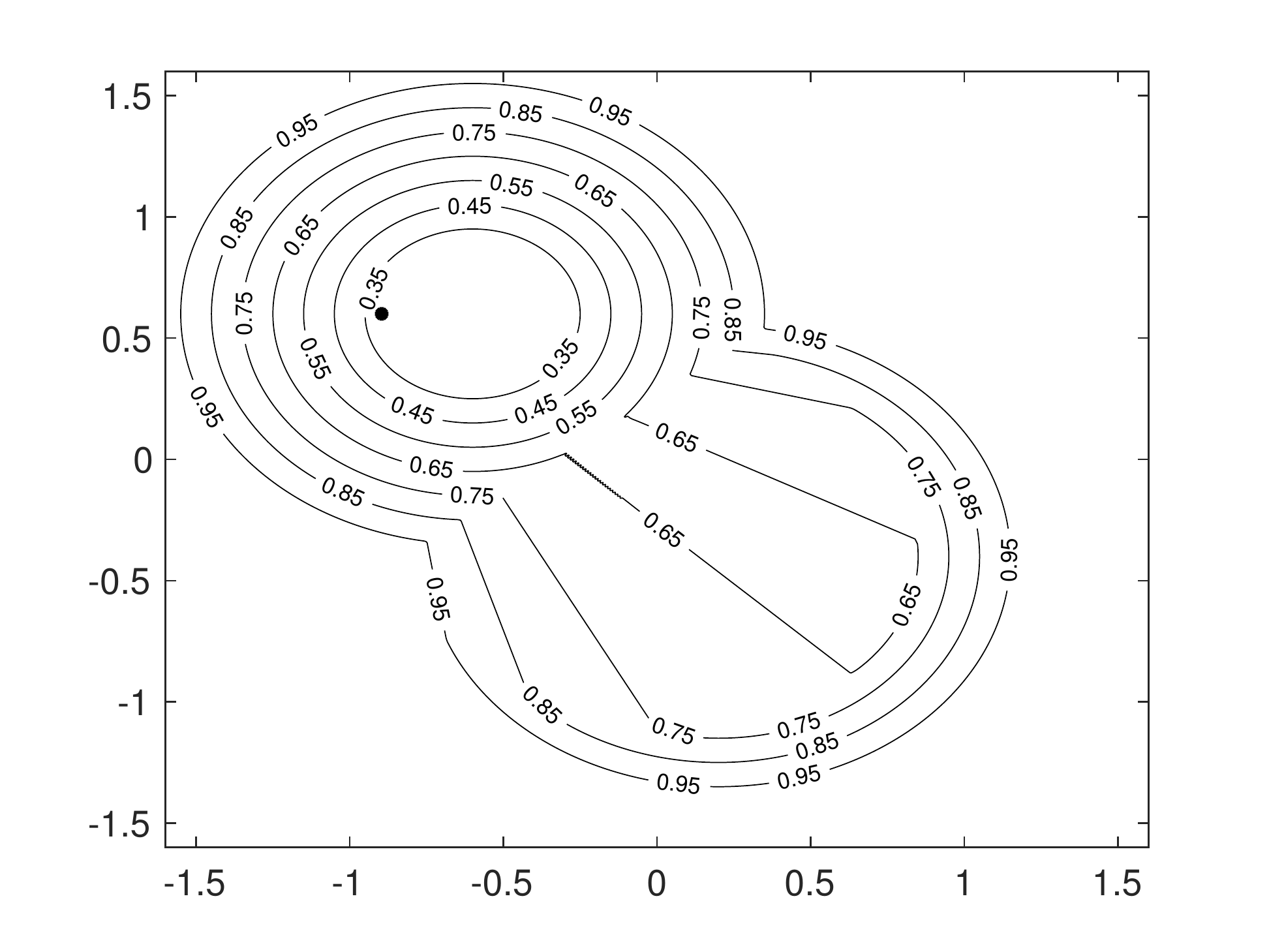}}
\caption{Contour plot of $g$ given in Example \ref{ex:regularity} (left), its lower star-shaped envelope $w$ (center) and its upper (visibility) envelope $u$ (right). The point $\xs$ is marked by $*$. The lower star-shaped envelope is discontinuous.}
    \label{fig:Ex1Envelope}
\end{figure}

\begin{figure}[h]
\subfigure{\includegraphics[width=\widththreefigures]{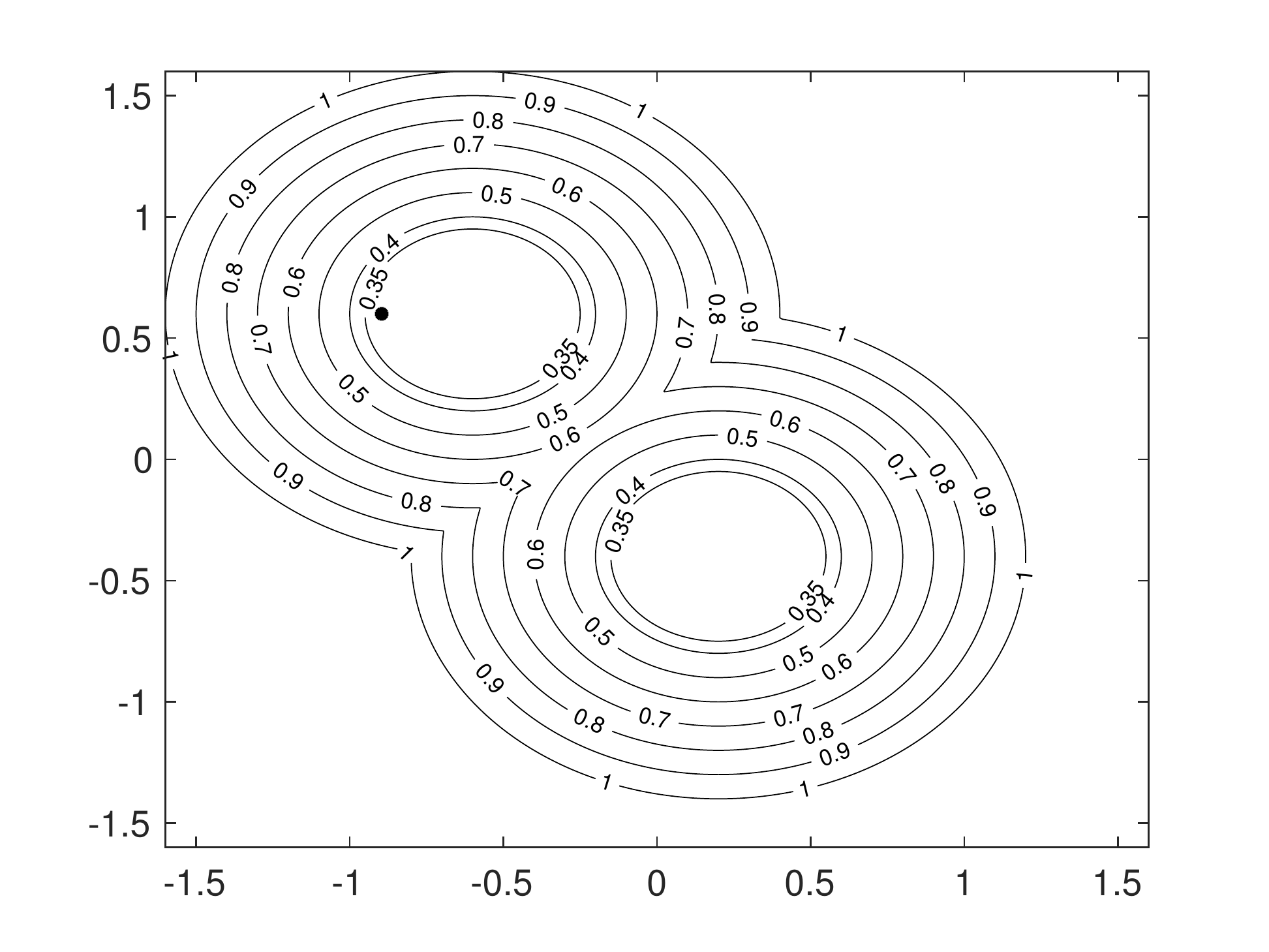}}
\subfigure{\includegraphics[width=\widththreefigures]{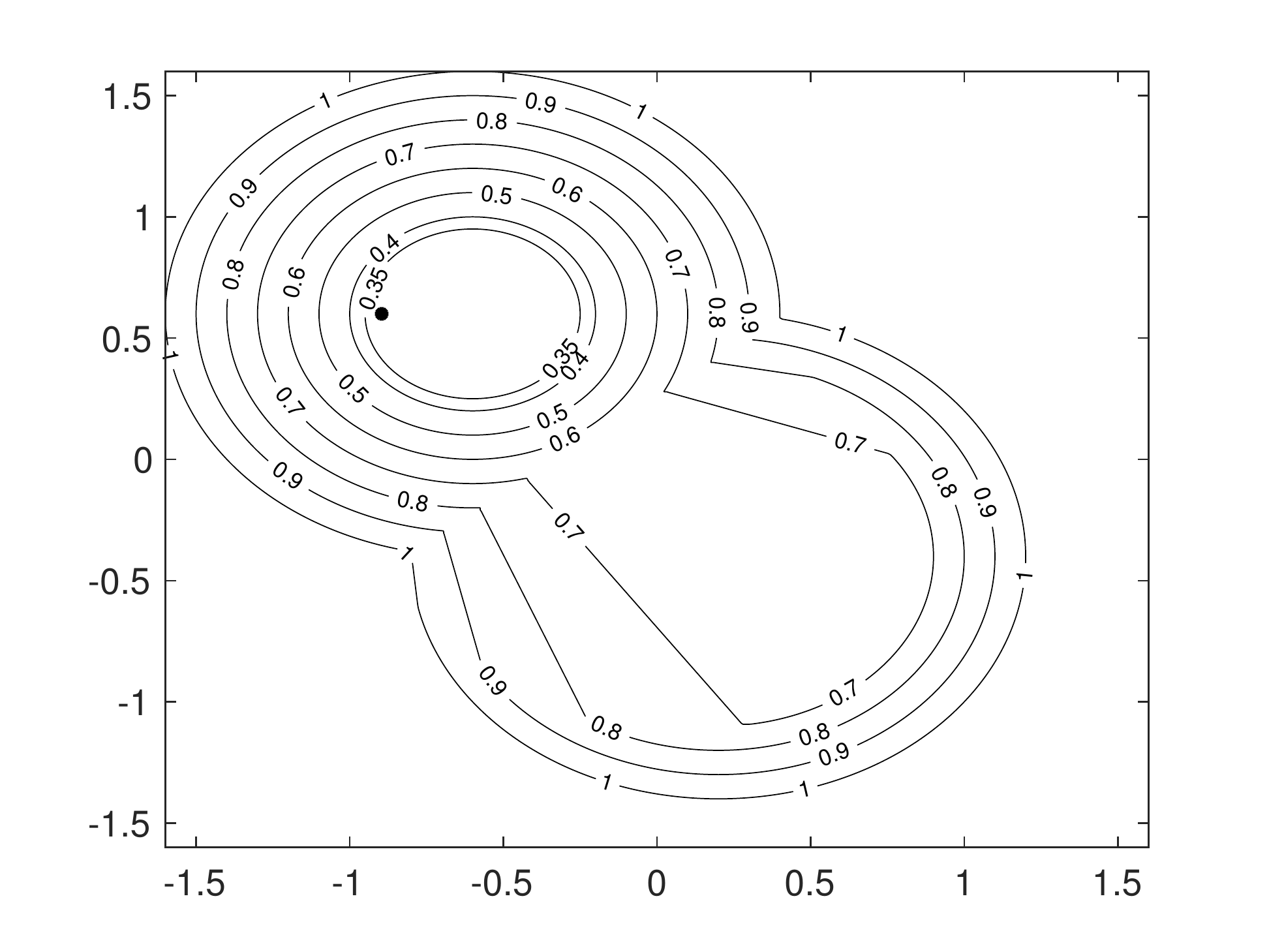}}
\subfigure{\includegraphics[width=\widththreefigures]{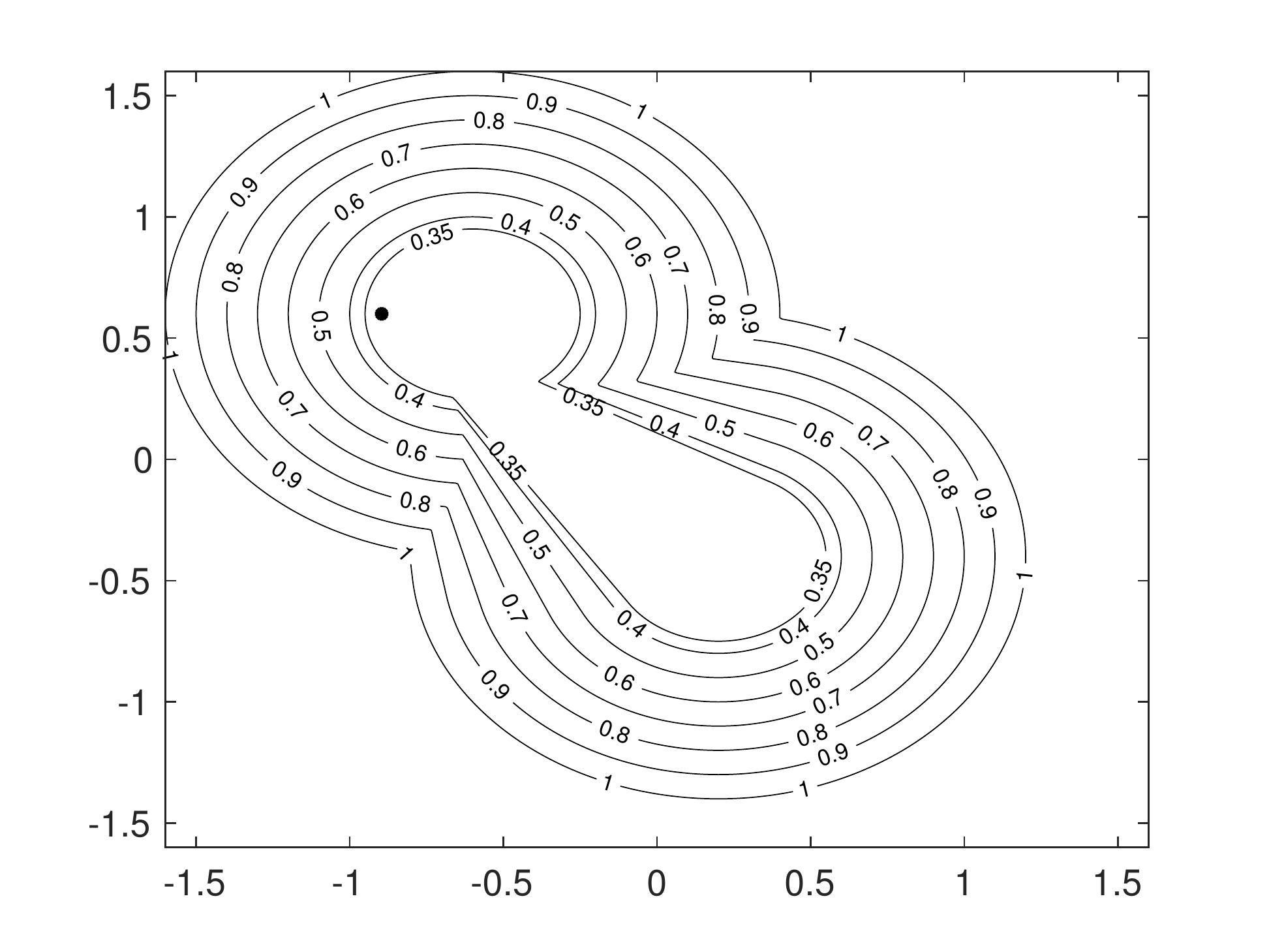}}
\caption{Contour plot of $\max(g,g(\xs)$ where $g$ given in Example \ref{ex:regularity} (left), its lower star-shaped envelope $w$ (center) and its upper (visibility) envelope $u$ (right). The point $\xs$ is marked by $*$. Since $\xs$ is a minimizer of $\max(g,g(\xs))$ both star-shaped envelopes are continuous.}
\label{fig:Ex1Envelope2}
\end{figure}

\subsection{A local PDE for star-shaped envelopes}

We start by establishing a first order condition for star-shaped functions.

\begin{proposition}\label{prop:SSfirstorder}
Suppose $u:\Rn\to\R$ is differentiable. Then $u$ is  star-shaped with respect to $\xs$ if and only if $\grad u(x) \cdot (x-\xs) \geq 0$ for all $x\in\R$.
\end{proposition}
\begin{proof}
By the mean value theorem, given any $t\in[0,1]$ there exists $s\in (0,t)$ such that
\[
u(t\xs + (1-t)x)-u(x) = t \grad u(x+s(\xs-x)) \cdot (\xs-x).
\]
Thus if $u$ is star-shaped with respect to to $\xs$ we obtain using Lemma \ref{lemma:SSequivalence}
\[
\grad u(x+s(\xs-x)) \cdot (\xs-x) \leq 0.
\]
Taking the limit as $t\to 0$ leads to the desired inequality $\grad u(x) \cdot (x-\xs) \geq 0$.

Now, suppose that $\grad u(x) \cdot (x-\xs) \geq 0$ for all $x \in \Rn$. We argue by contradiction. Assume that $u$ is not star-shaped with respect to $\xs$. Hence, by Lemma \ref{lemma:SSequivalence}, there are $y \in \Omega$ and $t\in[0,1]$ such that
\[
u(t \xs + (1-t)y) >  u(y).
\]
Then, again by the mean value theorem,
\[
\grad u(y+s(\xs-y)) \cdot (\xs-y) > 0
\]
for some $s\in (0,t)$. However, taking $x = y+s(\xs-y)$ in $\grad u(x) \cdot (x-\xs) \geq 0$ leads to
\begin{align*}
& \grad u(y+s(\xs-y)) \cdot (y+s(\xs-y)-\xs) \geq 0\\
\Longleftrightarrow & \grad u(y+s(\xs-y)) \cdot (y-\xs)(1-s) \geq 0\\
\Longleftrightarrow & \grad u(y+s(\xs-y)) \cdot (y-\xs) \geq 0\\
\Longleftrightarrow & \grad u(y+s(\xs-y)) \cdot (\xs-y) \leq 0.
\end{align*}
We have the desired contradiction and so the proof is complete.
\end{proof}

We are interested in star-shaped functions that may not be differentiable since typically the visibility set has corners. 
Thus we need to characterize star-shaped functions in the weak sense. We will do so using viscosity solutions. Throughout the rest of this section let $\Omega$ be a bounded star-shaped domain with respect to $\xs \in \Rn$ and denote its boundary by $\partial \Omega$.

\begin{proposition}\label{prop:SSviscosityequivalence}
Suppose $u\in USC(\Omega)$. Then $u$ is star-shaped with respect to $\xs$ if and only if $u$ is a viscosity subsolution of $\grad u(x) \cdot (\xs-x) = 0$.
\end{proposition}
\begin{proof}
Suppose $u$ is star-shaped with respect to $\xs$. Let $\phi \in C^1(\Omega)$ be such that $u-\phi$ has a local maximum at $x \in \Omega$. Without loss of generality assume that $\phi(x)=u(x)$. Then we have $u(y) \leq \phi(y)$ in a neighborhood of $x$. Since $u$ is star-shaped with respect to $\xs$, $u$ is increasing along arrays from $\xs$ to $x$ by Proposition \ref{prop:SSequivalence} and therefore
\[
u(x) \leq u(x+h(x-\xs))
\]
for $h > 0$. Hence, given the choice of $\phi$,
\[
\phi(x) \leq \phi(x+h(x-\xs))
\]
for $h > 0$ sufficiently small. Since
\[
\frac{\phi(x+h(x-\xs))-\phi(x)}{h} = \grad \phi(x+s(\xs-x)) \cdot (x-\xs) 
\]
for $s \in (0,h)$, we obtain $\grad\phi(x) \cdot (\xs-x) \leq 0$ as $h \to 0$. This shows that $u$ is a viscosity subsolution of $\grad u(x) \cdot (\xs-x) = 0$.

Suppose now that $u$ is a viscosity subsolution of $\grad u(x) \cdot (\xs-x) = 0$ and that $u$ is not star-shaped with respect to $\xs$.
Then there exists $y,z$ such that $u(y) \geq u(\xs)$ and $u(y) > u(z)$ with $y$ lying on the line segment from $\xs$ to $z$.
Without loss of generality assume that $y = \argmax_x u(x)$.
We can then construct a linear $\phi \in C^1(\Omega)$ such that $\grad \phi = (u(z)-u(y))(z-y)$ and $u-\phi$ has a local maximum at $x$ with $(x-\xs) \cdot (z-y) > 0$.
But then
\[
\grad\phi(x) \cdot (x-\xs) = (u(z)-u(y)) (z-y)\cdot (x-\xs) < 0
\]
which contradicts our assumption.
\end{proof}

We can now finally write the PDEs for the upper and lower star-shaped envelopes of $g$ with respect to $\xs$.

\begin{proposition}\label{prop:LowerEnv}
Let $g \in C(\Omega)$ be bounded by below and let $w = SS^-(g)$ be the lower star-shaped envelope $w$ of $g$ with respect to $\xs$. Assume $w$ is continuous. Then $w$ is the viscosity solution of the obstacle problem
\[
\max\{w(x)-g(x), (\xs-x)\cdot \grad w(x)\} = 0
\]
along with boundary conditions $w = g$ on $\partial \Omega$.
\end{proposition}
\begin{proof}
By Proposition \ref{prop:LowerEnvExplicit}, $w$ is star-shaped with respect to $\xs$ and therefore we can write
\[
w(x) = \sup \{ v(x) \mid \text{ $v\in USC(\Omega)$ is star-shaped with respect to $\xs$ and $v \leq g$}\}.
\]
Now, according to Proposition \ref{prop:SSviscosityequivalence}, $w$ is precisely the supremum of all subsolutions of the PDE and so the proof follows directly from Perron's method.
\end{proof}

\begin{proposition}\label{prop:UpperEnv}
Let $g\in C(\Omega)$ and let $u = SS^+(g)$ be the upper star-shaped envelope (visibility) of $g$ with respect to $\xs$. Then $u$ is the viscosity solution of the obstacle problem
\bq\label{PDE:visibility}
\min\{u(x)-g(x), (x-\xs)\cdot \grad u(x)\} = 0,
\eq
along with $u(\xs) = g(\xs)$.
\end{proposition}
\begin{proof}
By definition the upper star-shaped envelope is given by
\[
u = \inf  \{ v(x) \mid v(y)  \geq g(y) \text{ for all }y, \text{ $v$ is star-shaped with respect to $\xs$}  \}.
\]
We observe that this is equivalent to
\[
-u = \sup  \{ v(x) \mid v(y)  \leq -g(y) \text{ for all }y, \text{ $-v$ is star-shaped with respect to $\xs$} \}\\
\]
and so $-u$ is the solution of
\[
\max\{U-(-g),-(\xs-x)\cdot \grad U(x)\} = 0,
\]
by a similar reasoning to the one in Proposition \ref{prop:LowerEnv}. Now, since the equation can be rewritten as
\[
\min\{-U-g,(\xs-x)\cdot \grad U(x)\} = 0,
\]
we conclude that $u$ is the solution of
\[
\min\{u-g,(x-\xs)\cdot \grad u(x)\} = 0
\]
as desired.
\end{proof}

\begin{remark}
Contrary to Proposition \ref{prop:LowerEnv} we do not need to assume the continuity of the start-shaped envelope as this follows directly from the assumptions on $g$ (see Proposition \ref{prop:continuity}).
\end{remark}

\subsection{Comparison Principle}

An important property of elliptic equations, from which uniqueness follows, is the comparison principle that states that subsolutions lie below supersolutions. In addition, it also plays a crucial role when establishing the convergence of approximation schemes using the the theory of Barles and Souganidis \cite{BSnum}, which we intend to do later on. However, in such setting, the comparison principle required is a strong comparison principle: The boundary conditions are satisfied in the viscosity sense. In general, such a comparison is only available when the solutions are continuous up to the boundary (see \cite{CIL} for more details).


We focus our attention in PDE \eqref{PDE:visibility} as it is the one we are most interested in: Its solution allows us to determine the visibility set. A strong comparison principle is however not satisfied: Notice that any function that is nonincreasing along any direction away from $\xs$ is a subsolution and thus given any supersolution we can always construct a subsolution that lies above it by adding a large enough constant. We can however circumvent this by requiring that the subsolution $u$ satisfies $u(\xs) \leq g(\xs)$. This will also prove to be enough to establish convergence of our numerical scheme. Intuitively, imposing that the subsolution $u$ satisfies $u(\xs) \leq g(\xs)$ guarantees that $u$ lies below $g$. This, together with the fact that the solution of \eqref{PDE:visibility} is the infimum of all supersolutions according to Perron's method, is enough to reach the desired conclusion.

\begin{proposition}\label{prop:comparison}
Let $g \in C(\Omega)$. Let $u$ be a viscosity subsolution of \eqref{PDE:visibility} such that $u(\xs) \leq g(\xs)$ and let $v$ be a supersolution of \eqref{PDE:visibility}. Then $u \leq v$ in $\Omega$.
\end{proposition}
\begin{proof}
Suppose there exists $x \in \Omega$ such that $u(x) > v(x)$ by contradiction. Since $v$ is a supersolution \eqref{PDE:visibility},
\[
v(x) \geq SS^+(g)(x) = \max \{ g(y) \mid y = \xs + t(x-\xs) \in \Omega, t\in[0,1]\}
\]
where we used Perron's characterization of $SS^+(g)$ and Proposition \ref{prop:UpperEnvExplicit}. In particular, we have $v(x) \geq g(x)$ and $v(x) \geq g(\xs)$. Hence $u(x) > g(x)$ and $u(x) >~u(\xs)$ by assumption on $u$. Therefore there exists a linear function $\phi$ such that $u-\phi$ has a local maximum at $y$ with $u(y) > g(y)$ and $(y-\xs)\cdot \grad \phi(y) = u(x)-u(\xs)> 0$. We have derived a contradiction with the assumption that $u$ is a subsolution of \eqref{PDE:visibility} and the proof is complete.
\end{proof}

\subsection{Visibility from multiple viewpoints}

We are now interested in the visibility set from multiple viewpoints where a point is consider visible if is seen by at least one viewpoint. We follow the same ideas as before, starting by generalizing the definition of star-shaped set.

\begin{definition}\label{def:SSone}
We say that a set $S \subset \Rn$ is star-shaped with respect to $\{\xs_1,\ldots,\xs_r\}$ if
\[
y \in S \implies  \exists_{\xs\in\{\xs_1,\ldots,\xs_r\}} \forall_{t\in[0,1]} \: ty + (1-t)\xs \in S.
\]
\end{definition}

As in \autoref{sec:SS} we can define star-shaped functions with respect to $\{\xs_1,\ldots,\xs_n\}$ according to Definition \ref{def:SSone}. More importantly, Proposition \ref{prop:SSfirstorder} can be generalized.

\begin{proposition}
Suppose $u:\Rn\to\R$ is differentiable and bounded by below. Then $u$ is star-shaped with respect to $\{\xs_1,\ldots,\xs_r\}$ according to Definition \ref{def:SSone} if and only if
\[
\max_{i=1,\ldots,r} \min_{t\in[0,1]} \grad u(\xs_i+t(x-\xs_i)) \cdot (x-\xs_i) \geq 0.
\]
\end{proposition}

\begin{remark}
The presence of the minimum in $t$ may not appear obvious at first, but it is crucial here. In order for a point $x$ to be visible from $\xs_i$ then $u$ must be increasing along the ray from $\xs_i$ to $x$ which guarantees that all the points along the ray will be in the visible set. Without the minimum in $t$ a point $x$ could be consider visible by first moving along a ray towards $\xs_i$ and then follow a different viewpoint.
\end{remark}

In this case, $u = SS^+_{\{\xs_1,\ldots,\xs_r\}}(g)$ is the solution of the following PDE
\bq\label{PDE:visibility_alo}
\min\{u(x)-g(x), \max_{i = 1,\ldots,r}\min_{t\in[0,1]} \grad u(\xs_i+t(x-\xs_i)) \cdot (x-\xs_i)\} = 0.
\eq
Despite being a non-local PDE, unlike the previous ones with a single viewpoint $\xs$, we still have a fast solver available: The solution $u$ is given by
\[
u = \min_{i = 1,\ldots,r} u_i,
\]
where $u_i$ is the solution of 
\[
\min\{u(x)-g(x), (x-\xs_i)\cdot \grad u(x)\} = 0.
\]
This follows from noticing that
\[
SS^+_{\{\xs_1,\ldots,\xs_r\}}(g) = \min_{i = 1,\ldots,r} SS^+_{\xs_i}(g).
\]
In addition, the following solution formula is also available
\[
u(x) = \min_{i = 1,\ldots,r}\max \{ g(y) \mid y \in \xs_i + t(x-\xs_i) \in \Omega, t\in[0,1]\}.
\]

At this point, a natural question to ask is what other visibility definitions will lead to PDEs following the approach taken here. For instance, given $3$ viewpoints $\{\xs_1,\xs_2,\xs_3\}$ a point can be considered visible if (i) it is seen by at least two viewpoints; (ii) it is seen by $\xs_1$ or both $\xs_2$ and $\xs_3$ (iii) it is seen by all viewpoints. All these can be computed efficiently computed by first determining the visibility of each viewpoint and then taking the appropriate combination of maximums and minimums. However, determining a corresponding PDE following the approach considered here will in general fail. Consider for instance case (iii) where a point is visible if it is seen by all viewpoints. In this setting, one can no longer define star-shaped like sets as the visibility set may be disconnected (see Figure \ref{fig:ex_multiple_viewpoints} for a simple example with two viewpoints). The fundamental difference is that it is no longer true that if a point $x$ is visible then all points along the ray from $x$ to the viewpoint are visible.

\begin{figure}[htp]
\centering
\subfigure{\includegraphics[width=\widthtwofigures]{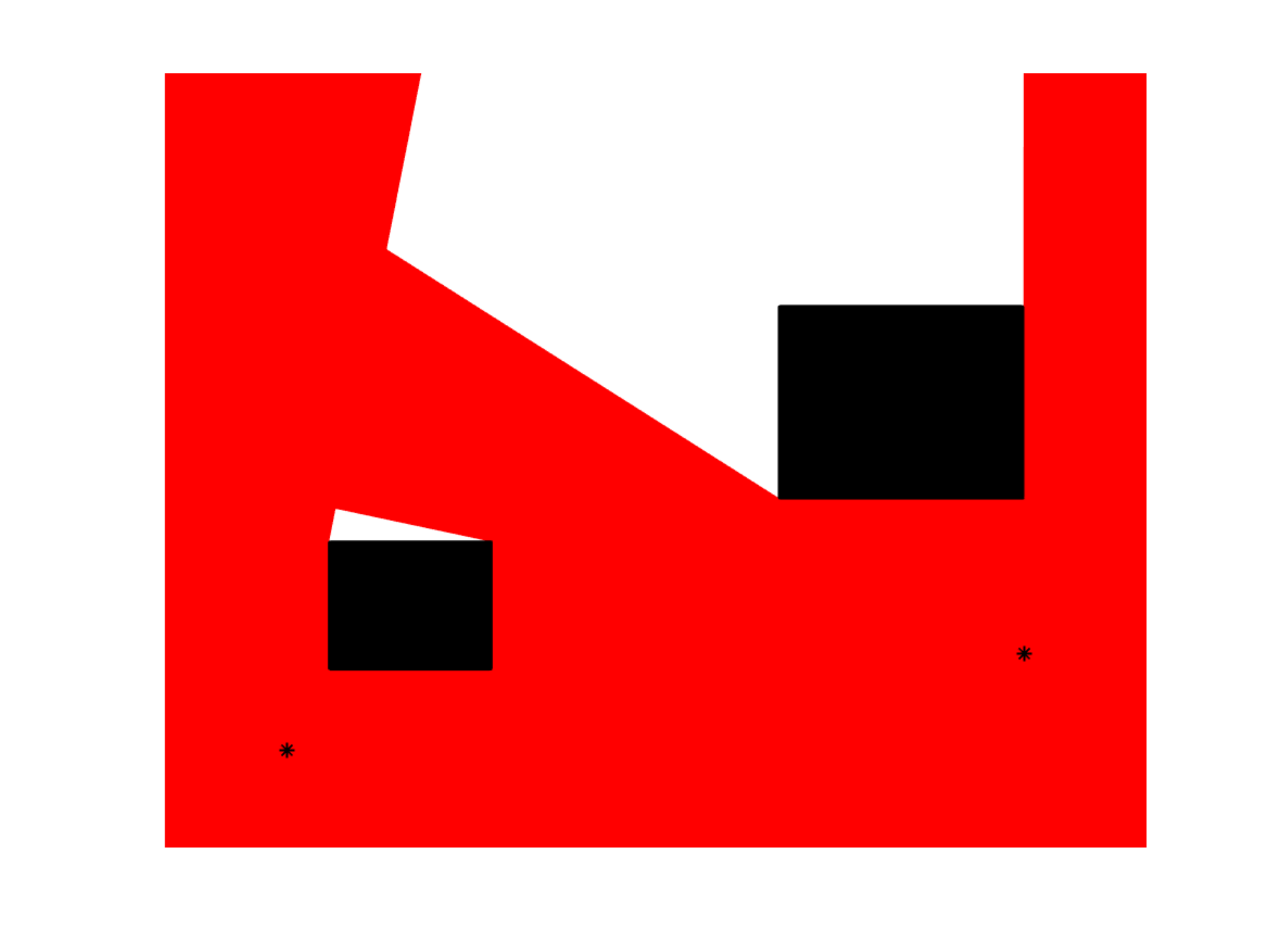}}
\subfigure{\includegraphics[width=\widthtwofigures]{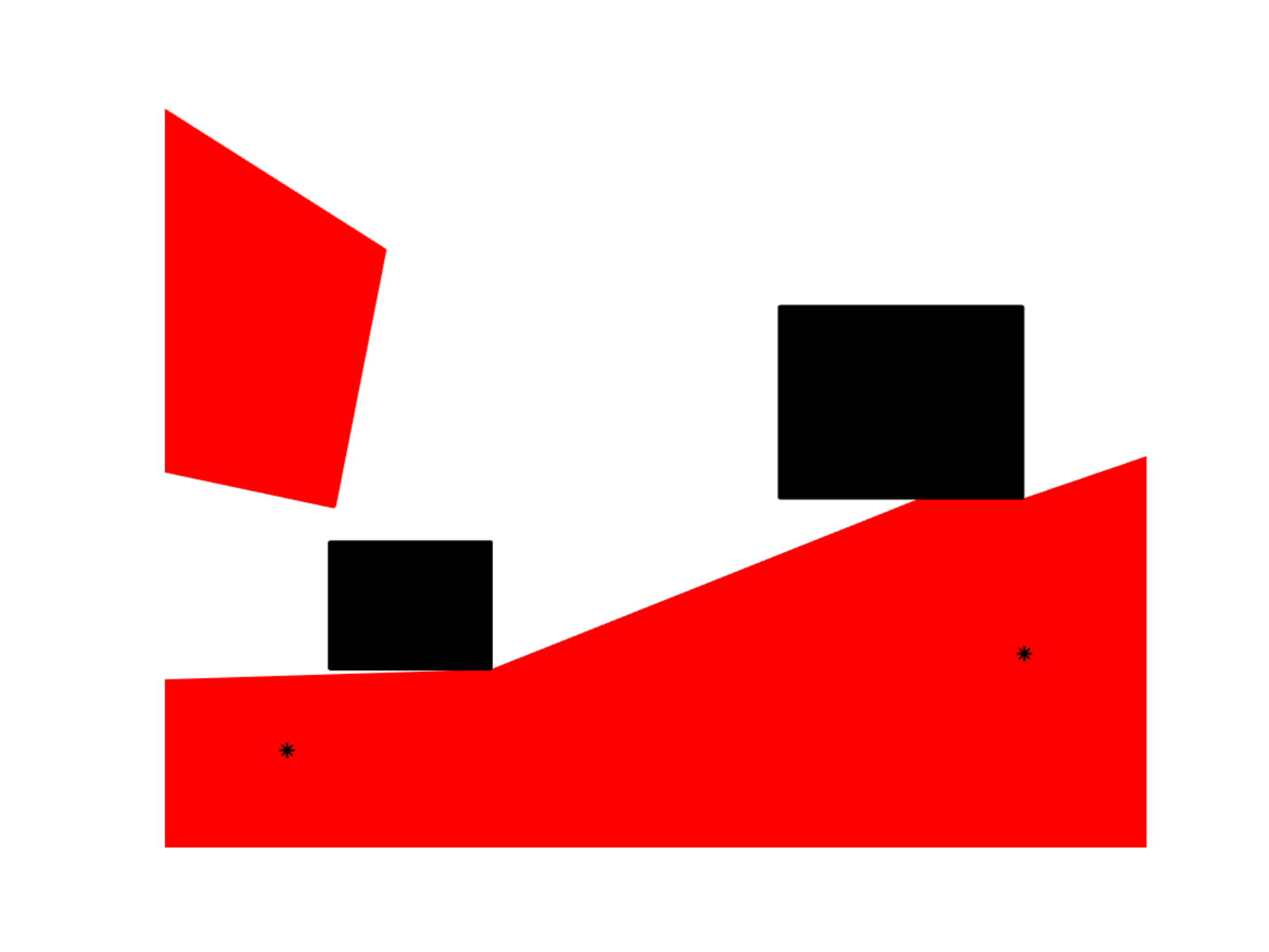}}
\caption{Visibility set from multiple viewpoints when a point is visible if it seen by at least one viewpoint (left) and by all the viewpoints (right).}
\label{fig:ex_multiple_viewpoints}
\end{figure}

\section{Convergent finite difference schemes}\label{sec:scheme}

In this section, we discuss the numerical schemes used to solve the PDEs introduced in the previous section, focusing our attention on the visibility PDE \eqref{PDE:visibility}.
As we will see, the scheme proposed here is degenerate elliptic finite difference schemes for which there exists a well established convergence framework.

Before we begin we introduce some notation. For simplicity, we will assume we are working on the hypercube $[-1, 1]^n \subseteq \Rn$. We write $x = (x_1, \ldots, x_n) \in [-1,1]^n$. The domain is discretized with a uniform grid, resulting in the following spatial resolution:
\[
h \equiv \frac{2}{N-1},
\]
where $N$ is the number of grid points used to discretize $[-1, 1]$. We denote by $\Omega^h$ the computational domain which in our case reduces to $[-1, 1]^n \cap h\Z^n$.

Our schemes are written as the operators $F^h[u] : C(\Omega^h) \to C(\Omega^h)$, where $C(\Omega^h)$ is the set of grid functions $u : \Omega^h \to \R$. We assume they have the following form:
\[
F^h[u](x) = F^h(u(x), u(x) - u(\cdot)) \quad \text{for } x \in \Omega^h
\]
where $u(\cdot)$ corresponds to the value of u at points in $\Omega^h$.

\subsection{Finite difference for the visibility PDE}

For simplicity, we present the scheme in the two-dimensional setting. The generalization to higher dimensions is straightforward.

Let $x \in \Omega^h$. Define the vector $\vs = x-\xs$ for $x\in\Omega^h$. If $x$ is one of the four grid points enclosing $\xs$, let $\tilde{x} = \xs$. Otherwise, let $\tilde{x}$ denote the intersection between the line that passes through $x$ and $\xs$ and a line segment formed by 2 of the 8 neighbors of $x$. 
Then $u_{\vs}(x) = (x-\xs) \cdot \grad u(x)$ and its upwind approximation is given by
\[
u^h_v(x) \equiv \frac{u(x)-\mathcal{I}_h u(\tilde{x})}{\abs{x-\tilde{x}}},
\]
where $\mathcal{I}_h$ is the piecewise linear Lagrange interpolant. The numerical scheme for the visibility PDE \eqref{PDE:visibility} is then given by
\bq\label{scheme:visibility}
F^h[u](x) = \min\{u(x)-g(x), u^h_v(x)\},	 \quad x \in \Omega^h.
\eq

\subsection{Fast sweeping solver}

We implement a fast sweeping solver to compute the solutions of \eqref{scheme:visibility}. Solving the equation $F^h[u](x) = 0$ for the reference variable, $u(x)$, leads to the update formula
\[
u(x) = \max\{g(x),\mathcal{I}_h u(\tilde{x})\},
\]
where $\tilde{x}$ was defined in the previous section. Since all characteristic are straight lines and flow away from $\xs$, the domain $\Omega^h$ only needs to be swept once (but in a very specific ordering). For simplicity, we will assume that $\xs \in \Omega^h$. Let $u_{i,j}$ denote the solution at $x_{i,j}$ where
\[
x_{i,j} = (-1+(i-1)h,-1+(j-1)h).
\]
Set $(i^*,j^*)$ such that $x_{i^*,j^*} \leq \xs < x_{i^*+1,j^*+1}$ (here the inequalities are interpreted component-wise). Set as well $u_{i^*,j^*} = g(\xs)$. The domain is then divided into four quadrants and each is swept in the following way:
\begin{itemize}
	\item $i = i^*,\ldots,N$, $j = j^*,\ldots, 1$ (sweeping bottom right square)
	\item $i = i^*,\ldots,N$, $j = j^*,\ldots, N$ (sweeping top right square) 
	\item $i = i^*+1,\ldots,1$, $j = j^*,\ldots, 1$ (sweeping bottom left square)
	\item $i = i^*+1,\ldots,1$, $j = j^*,\ldots, N$ (sweeping top left square)
\end{itemize}

\begin{remark}
Alternatively, the solver can be seen as a direct consequence of the solution formula \eqref{eq:explicit_visibility}.
\end{remark}

\section{Convergence of Numerical Solutions}\label{sec:convergence}

In this section we recall the notion of degenerate elliptic schemes and show that the solutions of the proposed numerical scheme converges to the solutions of \eqref{PDE:visibility} as the discretization parameter tends to zero. The standard framework used to establish convergence is that of Barles and Souganidis \cite{BSnum}, which we state below. In particular, it guarantees that the solutions of any monotone, consistent, and stable scheme converge to the unique viscosity solution of the PDE.

\subsection{Degenerate elliptic schemes}

Consider the Dirichlet problem for the degenerate elliptic PDE, $F[u] = 0$, and recall its corresponding finite difference formulation:
\[
\begin{cases}
F_h(x,u(x), u(x)-u(\cdot)) = 0,	& x \in \Omega^h,\\
u(x)-g(x) = 0,				& x \in \boundary^h,
\end{cases}
\]
where $h$ is the discretization parameter.

\begin{definition}
$F_h[u]$ is a degenerate elliptic scheme if it is non-decreasing in each of its arguments.
\end{definition} 

\begin{remark}
Although the convergence theory in \cite{BSnum} is originally stated in terms of monotone approximation schemes (schemes with non-negative coefficients), ellipticity is an equivalent formulation for finite difference operators \cite{ObermanSINUM}.
\end{remark}

\begin{definition}
The finite difference operator $F_h[u]$ is consistent with $F[u]=0$ if for any smooth function $\phi$ and $x\in\Omega$,
\begin{align*}
\lim_{h \to 0,y \to x, \xi \to 0} F_h(y,\phi(y)+\xi,\phi(y)-\phi(\cdot)) = F(x,\phi(x),\nabla \phi(x)).
\end{align*}
\end{definition}

\begin{definition}
The finite difference operator $F_h[u]$ is stable if there exists $M > 0$ independent of $h$ such that if $F_h[u] = 0$ then $\norm{u}_\infty \leq M$.
\end{definition}
\begin{remark}[Interpolating to the entire domain]
The convergence theory assumes that the approximation scheme and the grid function are defined on all of $\Omega$. Although the finite difference operator acts only on functions defined on $\Omega^h$, we can extend such functions to $\Omega^h$ via piecewise linear interpolation. In particular, performing piecewise linear interpolation maintains the ellipticity of the scheme, as well as all other relevant properties. Therefore, we can safely interchange $\Omega$ and $\Omega^h$ in the discussion of convergence without any loss of generality
\end{remark}

\subsection{Convergence of numerical approximations}

Next we will state the theorem for convergence of approximation schemes, tailored to elliptic finite difference schemes, and demonstrate that the proposed scheme fits in the desired framework. In particular, we will show that the schemes are elliptic, consistent, and have stable solutions.

\begin{proposition}[Convergence of approximation schemes \cite{BSnum}]\label{prop:convergence}
Let $u$ denote the unique viscosity of the degenerate elliptic PDE $F[u]=0$ with Dirichlet boundary conditions for which there exists a strong comparison principle. For each $h$, let $u^h$ denote the solutions of $F^h[u]=0$, where the finite difference scheme $F_h[u]$ is a consistent, stable and elliptic scheme. Then $u^h \to u$ locally uniformly on $\Omega$ as $h \to 0$.
\end{proposition}

Now, we check that our scheme is consistent, stable and elliptic.

\begin{lemma}[Consistency]\label{lemma:consistency}
The scheme is consistent.
\end{lemma}
\begin{proof}
It is sufficient to show that
\[
(x-\xs) \cdot \grad u(x) = u^h_v(x) + \bO(h)
\]
which follows immediately from a Taylor expansion argument and the definition of $u^h_v$.
\end{proof}

\begin{lemma}[Stability]\label{lemma:stability}
Suppose $g$ is bounded. Then the scheme is stable.
\end{lemma}
\begin{proof}
Since the solution $u^h$ satisfies
\[
u^h(x) = \max\{g(x),\mathcal{I}_h u^h(\tilde{x})\}, \quad x \in \Omega^h,
\]
it follows that $g(\xs) \leq u^h(x) \leq \norm{g}$.
\end{proof}

\begin{lemma}[Ellipticity]\label{lemma:ellipticity}
The scheme is elliptic.
\end{lemma}
\begin{proof}
The term $u(x)-g(x)$ is trivially elliptic. Since $\tilde{x}$ belongs to the line segment of two grid points, $\mathcal{I}_h u(\tilde{x})$ is a convex combination of neighboring grid values and therefore $u^h_v$ is elliptic. Hence $F^h$, being the minimum of two elliptic schemes, is also elliptic.
\end{proof}

\begin{theorem}
Suppose $g \in C(\Omega)$. Then the solutions of \eqref{scheme:visibility} converge locally uniformly on $\Omega$ as $h\to 0$ to the unique viscosity solution of \eqref{PDE:visibility} along with $u(\xs) = g(\xs)$.
\end{theorem}
\begin{proof}
By Lemmas \ref{lemma:consistency}, \ref{lemma:stability} and \ref{lemma:ellipticity}, we see that $F^h$ \eqref{scheme:visibility} is consistent, stable and elliptic.

In order to apply Proposition \ref{prop:convergence}, our PDE must satisfy a strong comparison principle which is not the case here. However, inspecting the proof of Proposition \ref{prop:convergence} in \cite{BSnum}, one notices that it is enough to show that $\overline{u} \leq \underline{u}$ in $\Omega$, where
\[
\overline{u}(x) = \limsup_{h\to 0,y\to x} u^h(x) \text{ and } \underline{u}(x) = \liminf_{h\to 0,y\to x} u^h(x).
\]
This will follow from Proposition \ref{prop:comparison} by proving that $\overline{u}(\xs) \leq g(\xs)$. Indeed, due to the continuity of $g$,
\begin{align*}
\overline{u}(\xs)	& = \limsup_{h\to 0,y\to \xs} u^h(\xs)\\
					& = \lim_{\epsilon\to 0} \sup\{u^h(y): y \in B(\xs,\epsilon)\setminus\{\xs\}, 0 < h < \epsilon\}\\	
					& \leq \lim_{\epsilon\to 0} \sup_{B(\xs,\epsilon)} g(x)\\
					& = g(\xs).
\end{align*}


\end{proof}

\section{Numerical results}\label{sec:numerics}

In this section we present the numerical results. We focus on solving equation \eqref{PDE:visibility} as it is the one we are mainly interested since the sublevel sets of its solution provide us with the visibility set. We present results both in two and three dimensions.

\begin{example}\label{ex_convergence}
We start with a simple example to show a numerical convergence test. We take $\xs=(-1,-1)$ as the viewpoint and consider as the obstacle function the cone $g(x) = -\sqrt{x_1^2+x_2^2}$ and therefore each level-set of $g$ corresponds to a circle with center at the origin and a different radius. The exact solution was obtained using \eqref{eq:explicit_visibility}. The difference between the numerical solution and the exact solution in the $l_\infty$ norm is presented in Table \ref{table:ex_convergence}, where we also confirm the expected first order convergence. Moreover, in Figure \ref{fig:ex_convergence}, we plot the level sets of both $g$ and the numerical solution, as well as their respective surface plots. As we can see, the visibility set is computed for each level set of $g$.
\end{example}

\begin{figure}[htp]
\centering
\subfigure{\includegraphics[width=\widththreefigures]{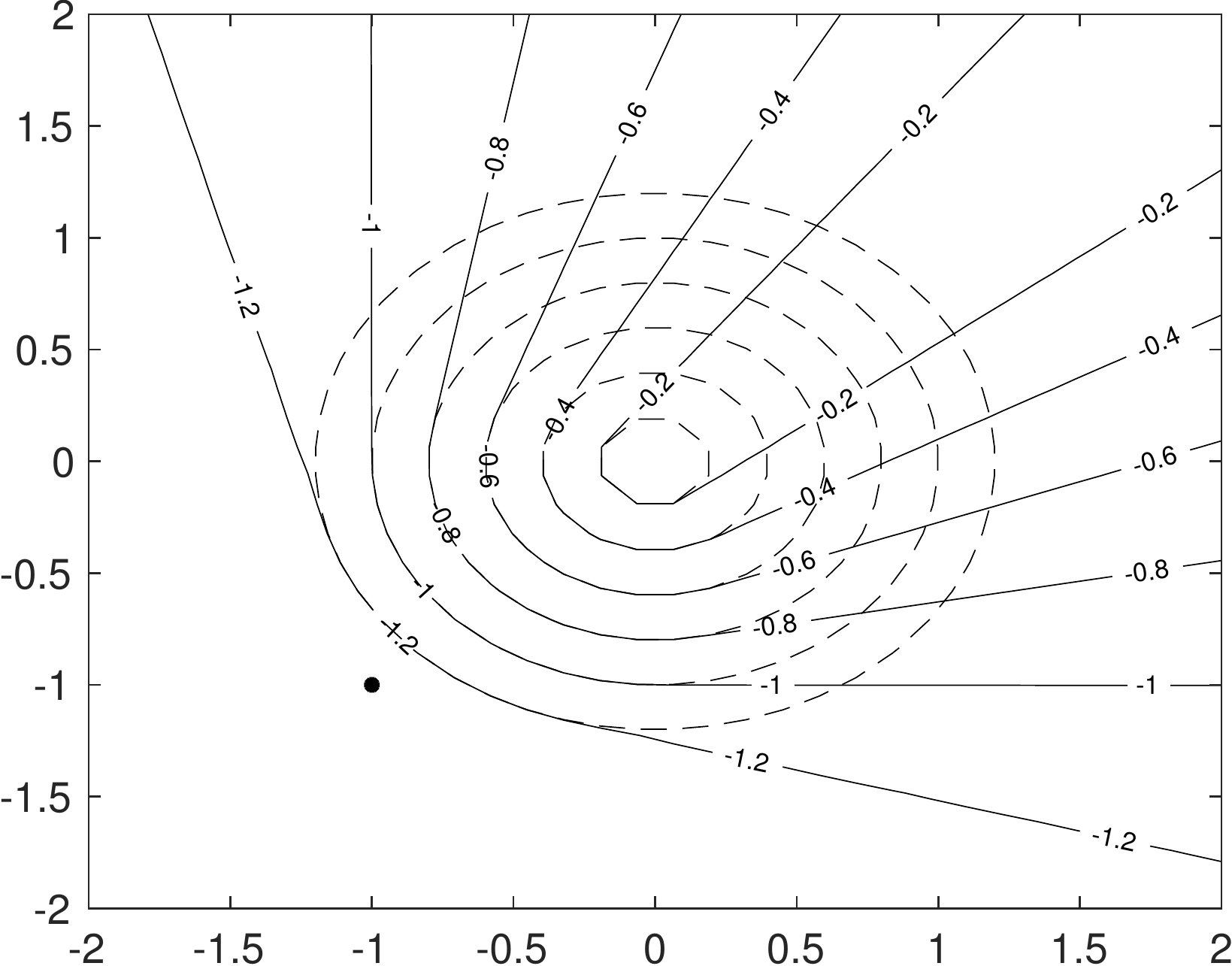}}
\subfigure{\includegraphics[width=\widththreefigures]{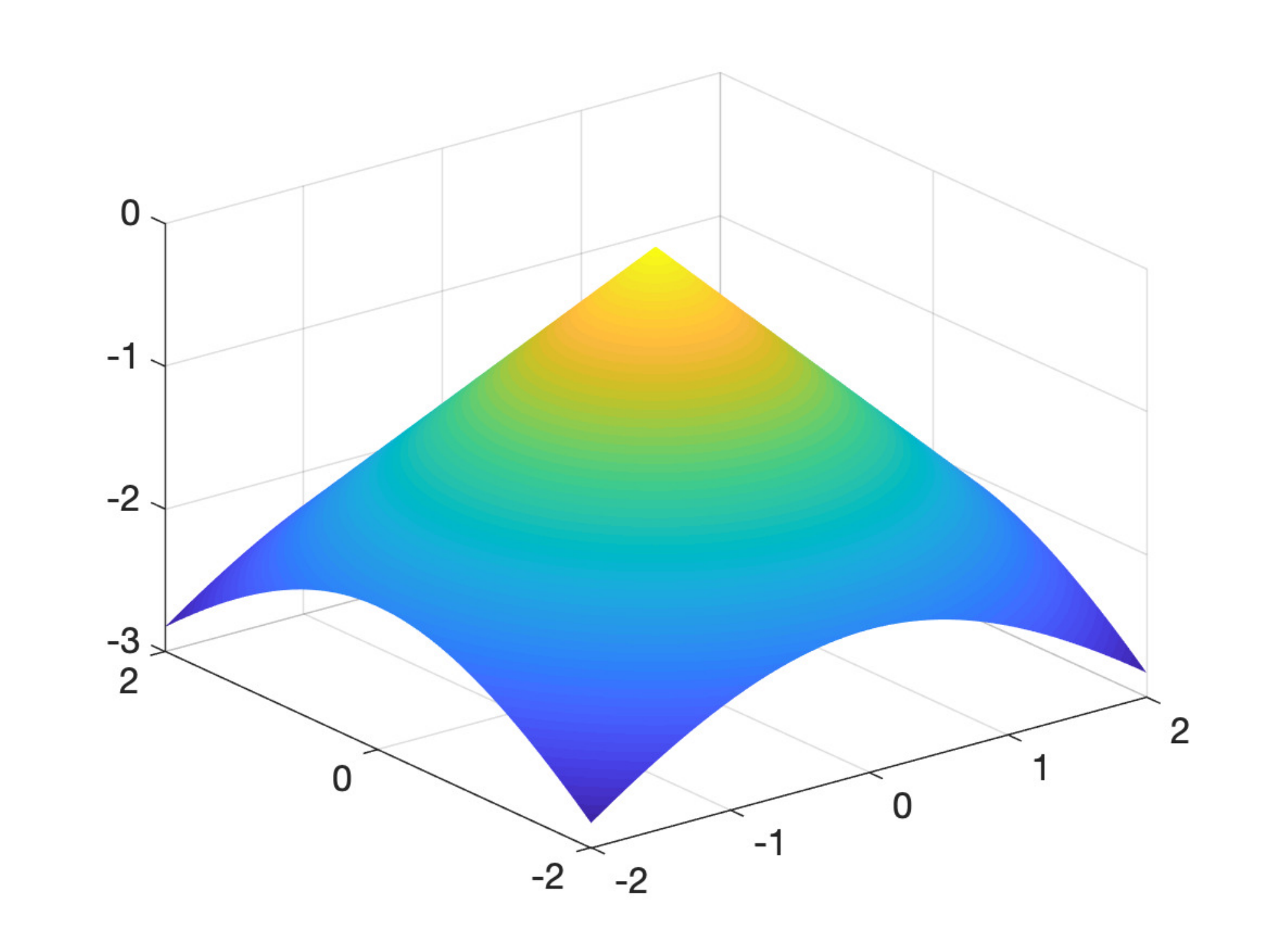}}
\subfigure{\includegraphics[width=\widththreefigures ]{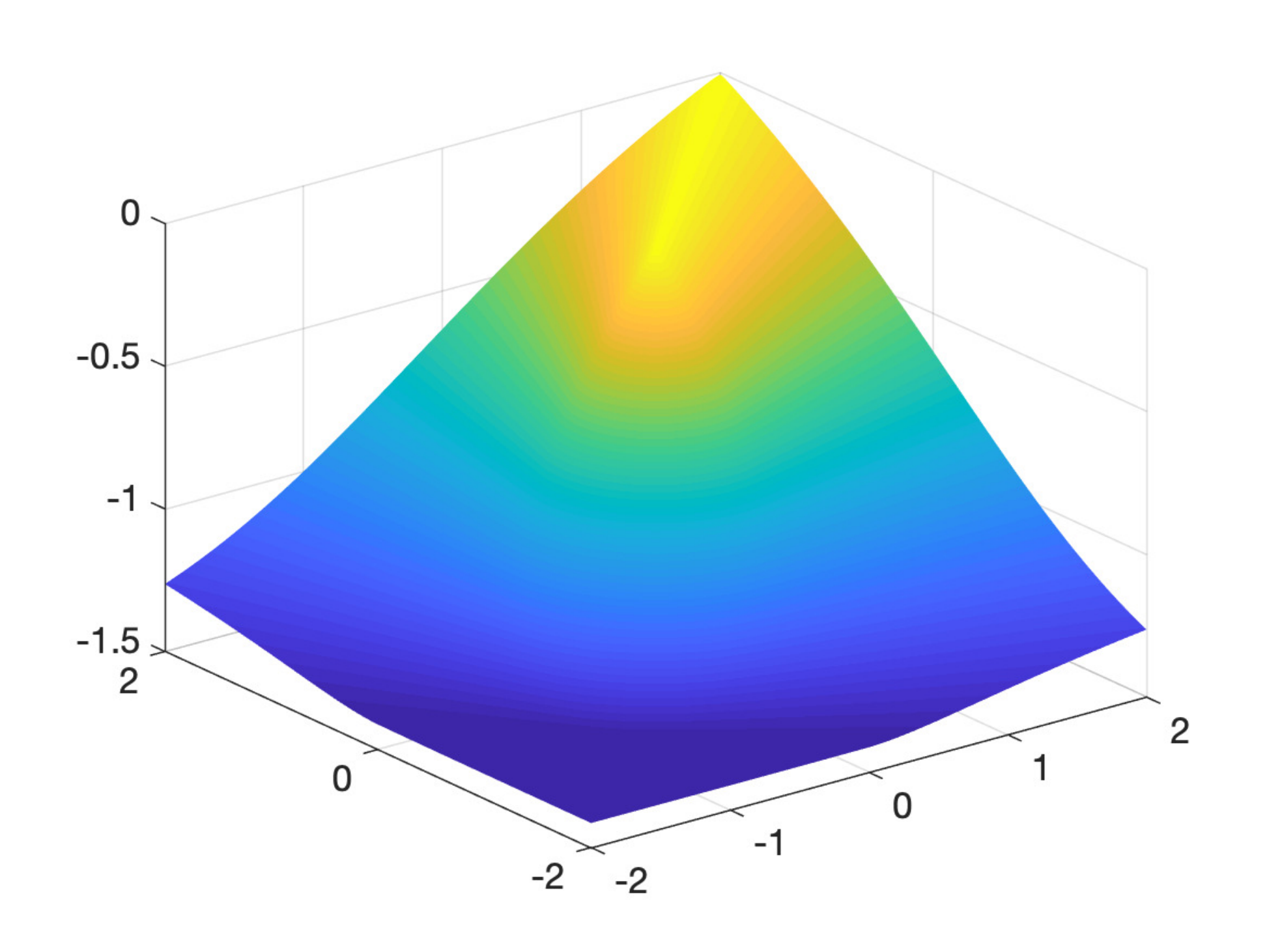}}
\caption{Level sets of the obstacle $g$ and the solution $u$ of \eqref{PDE:visibility} (left) and the respective surface plots (center and right).}
\label{fig:ex_convergence}
\end{figure}

\begin{table}[h]
\centering
\begin{tabular}{cccc}
N & $h$ & Error & Order\\
\hline
32 & \num{1.29e-01} & \num{9.12e-02} & - \\
64 & \num{6.35e-02} & \num{4.49e-02} & 1.02 \\
128 & \num{3.15e-02} & \num{2.23e-02} & 1.01 \\
256 & \num{1.57e-02} & \num{1.11e-02} & 1.01 \\
512 & \num{7.83e-03} & \num{5.54e-03} & 1.00 \\
1024 & \num{3.91e-03} & \num{2.76e-03} & 1.00 \\
2048 & \num{1.95e-03} & \num{1.38e-03} & 1.00 \\
4096 & \num{9.77e-04} & \num{6.91e-04} & 1.00 \\
\hline
\end{tabular}
\caption{Errors and order of convergence for Example \ref{ex_convergence}.}
\label{table:ex_convergence}
\end{table}


\begin{example}\label{ex2}
In this example we are interested in computing the visibility set where we have four different obstacles: two squares with centers $(-1.5,-0.2)$, $(0,0.3)$ and side lengths $0.5$, $1$ respectively and two circles with origins $(-0.3,1.5)$, $(-0.3,-1.4)$ and radius $0.5$. We achieve this by considering the obstacle function
\[
g(x) = -\min\left\{g_1(x),g_2(x),g_3(x),g_4(x)\right\}
\]
where
\begin{align*}
g_1(x) & = 2\max\{\abs{x_1+1.5},\abs{x_2+0.2}\},\\
g_2(x) & = \max\{\abs{x_1},\abs{x_2-0.3}\},\\
g_3(x) & = \sqrt{(x_1+0.3)^2+(x_2-1.5)^2},\\
g_4(x) & = \sqrt{(x_1+0.3)^2+(x_2+1.4)^2},
\end{align*}
and looking into the $0.5$ level set. We first solve  \eqref{PDE:visibility} with $\xs = (-1.5,-1.4)$ and $\xs = (1.5,-0.3)$. We compute as well the visibility set with respect to $\{\xs_1,\xs_2\} = \{(-1.5,-1.4),(1.5,-0.3)\}$ when a point is consider visible if it is seen by any of the viewpoints and by both viewpoints simultaneous. The former corresponds to the solution of \eqref{PDE:visibility_alo}. All results are displayed in Figure \ref{fig:ex2}.

\end{example}


\begin{figure}[htp]
\centering
\subfigure{\includegraphics[width=\widthtwofigures]{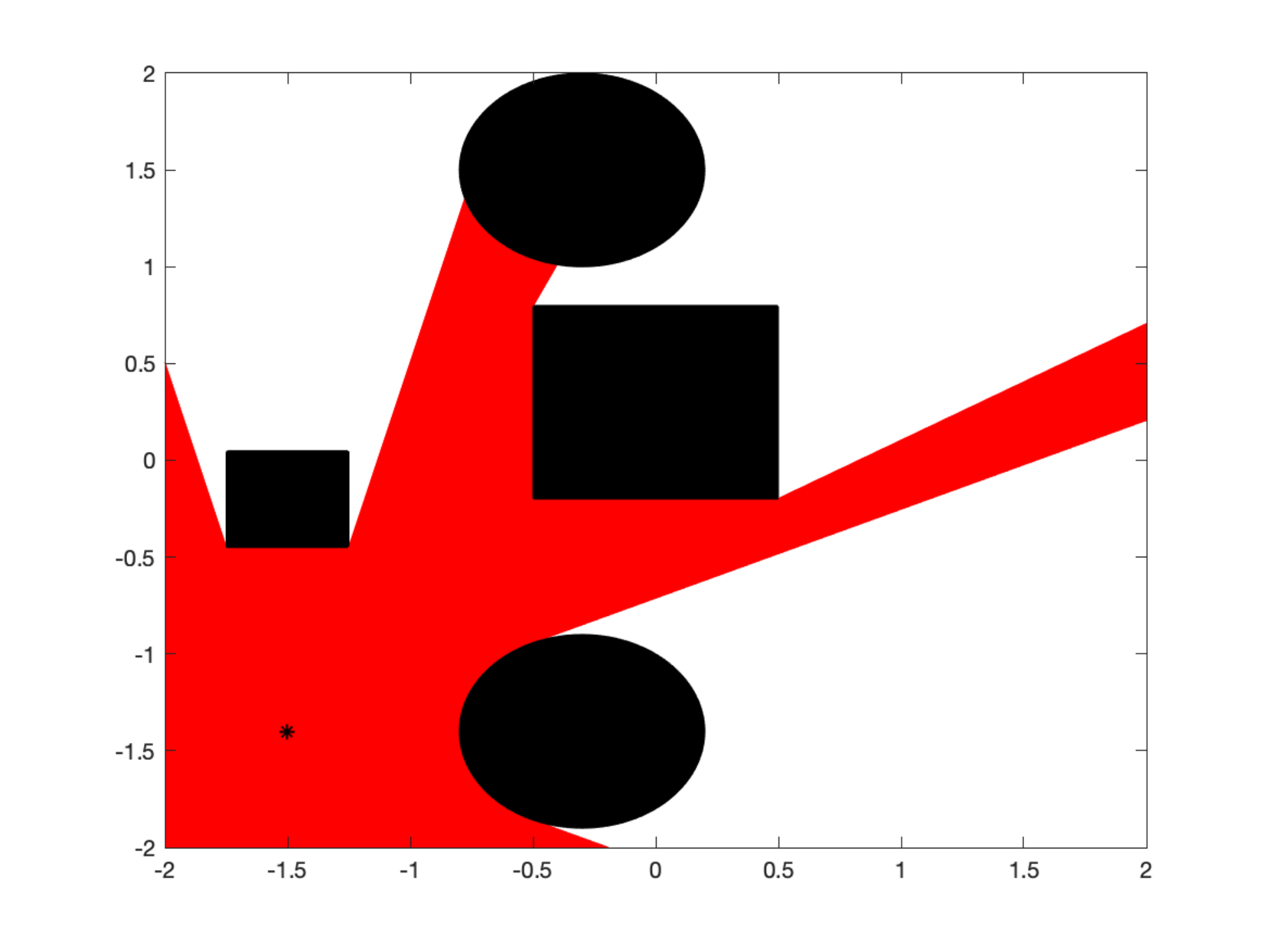}}
\subfigure{\includegraphics[width=\widthtwofigures]{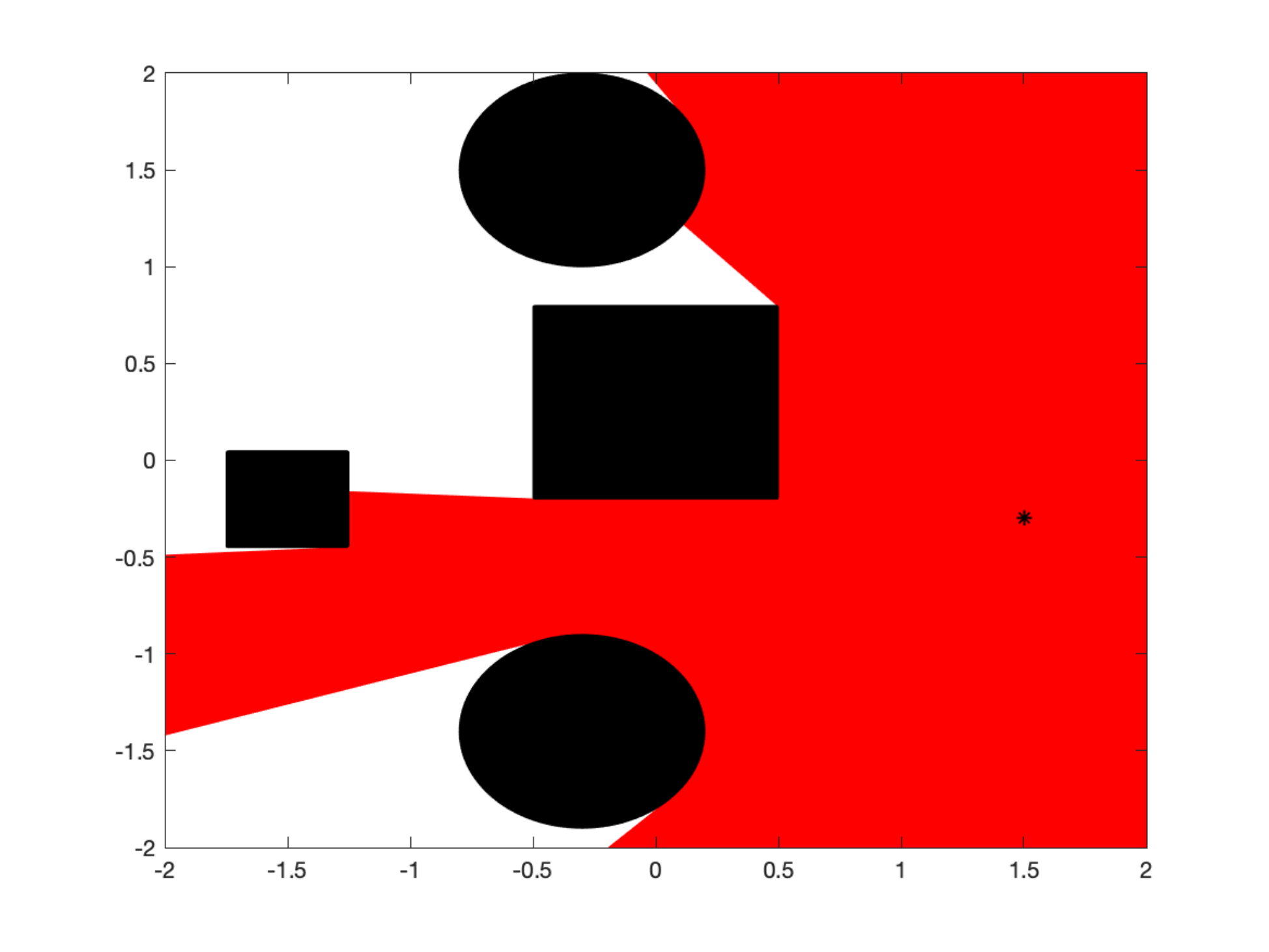}}
\subfigure{\includegraphics[width=\widthtwofigures]{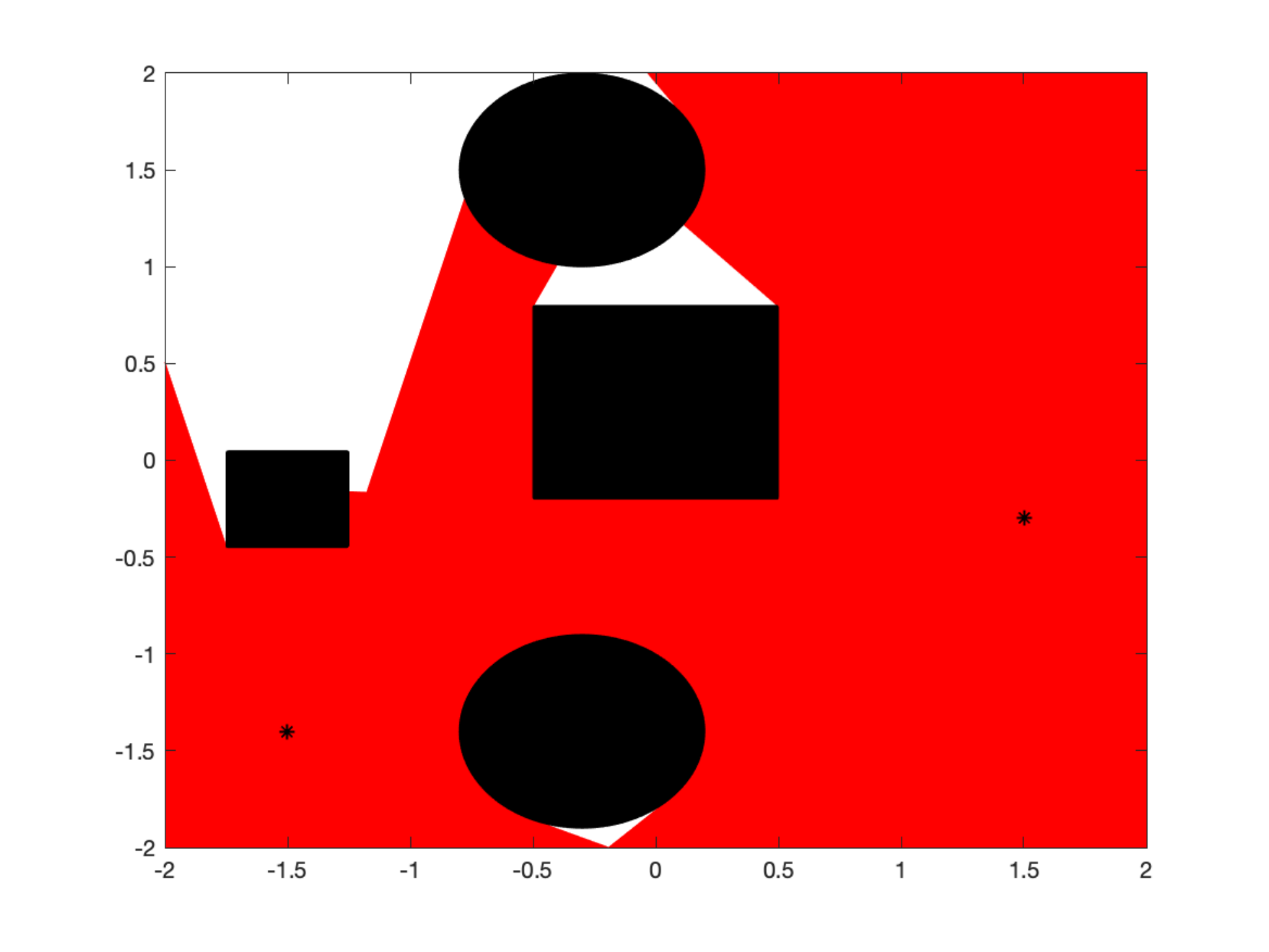}}
\subfigure{\includegraphics[width=\widthtwofigures]{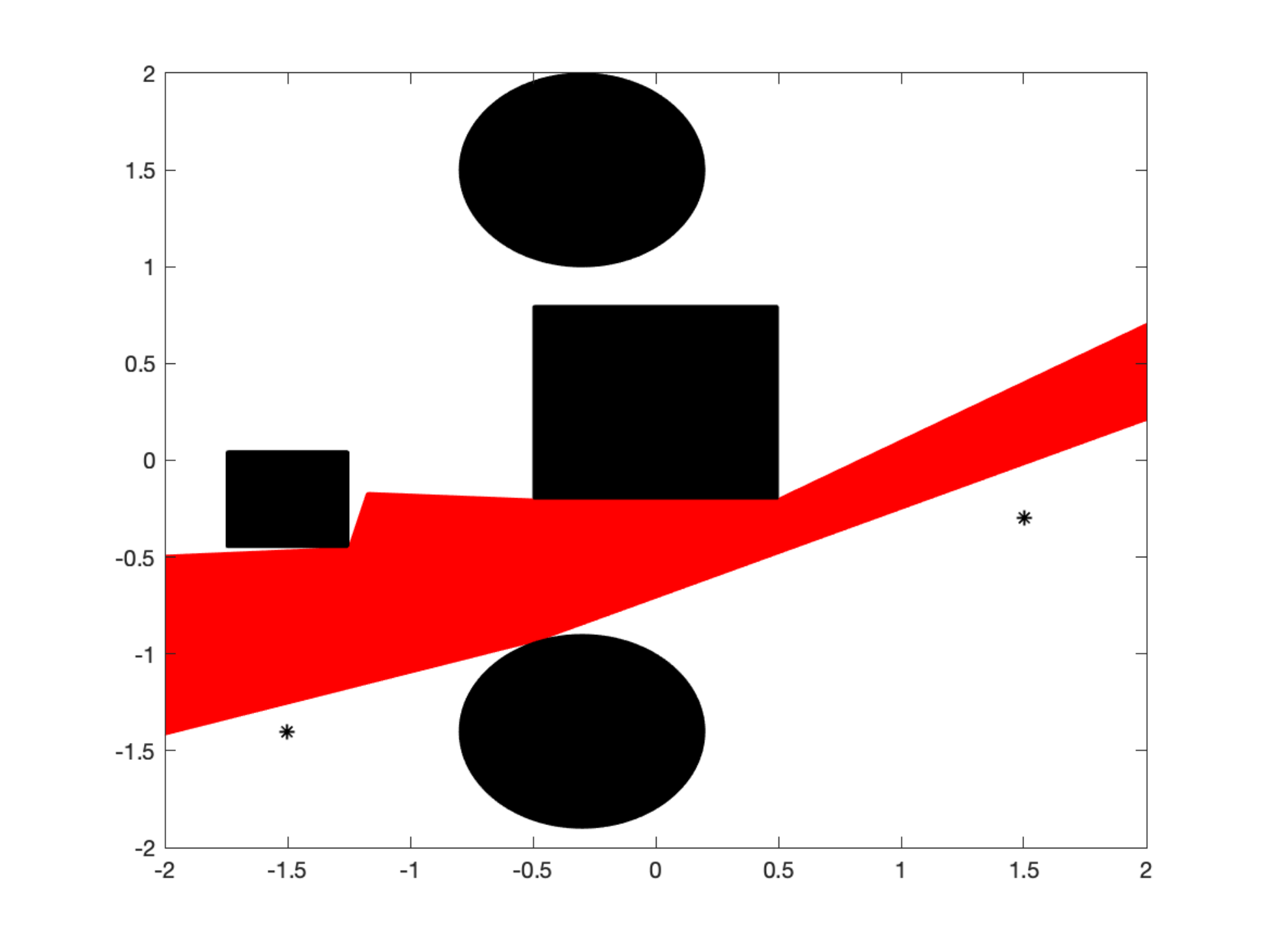}}
\caption{Results for Example \ref{ex2}: solution to \eqref{PDE:visibility} with $\xs = (-1.5,-1.4)$ (top-right); solution to \eqref{PDE:visibility} with $\xs = (1.5,-0.3)$ (top-left); solution to \eqref{PDE:visibility_alo} with $\{\xs_1,\xs_2\} = \{(-1.5,-1.4),(-1.5,-1.4)\}$ (bottom-left), i.e., set of points visible by any of the viewpoints; set of points visible by both viewpoins (bottom-right). All visibility sets are displayed in red, while the obstacles are displayed in black.}
\label{fig:ex2}
\end{figure}

\begin{example}\label{ex3}
We consider a simple three-dimensional example where the obstacle function is given by
\[
g(x) = -\min\{g_1(x),g_2(x)\}
\]
where
\[
g_1(x) = \max(\abs{x_1+2},\abs{x_2},x_3)-1 \quad\text{and}\quad g_2(x) = \max(\abs{x_1-3}),\abs{x_2-4},x_3)-2.
\]
It can be interpreted as the visibility set of a $360^\circ$ camera in the middle of two buildings.
The results are displayed in Figure \ref{fig:ex3}.
\end{example}

\begin{figure}[htp]
\centering
\subfigure{\includegraphics[width=\widthtwofigures]{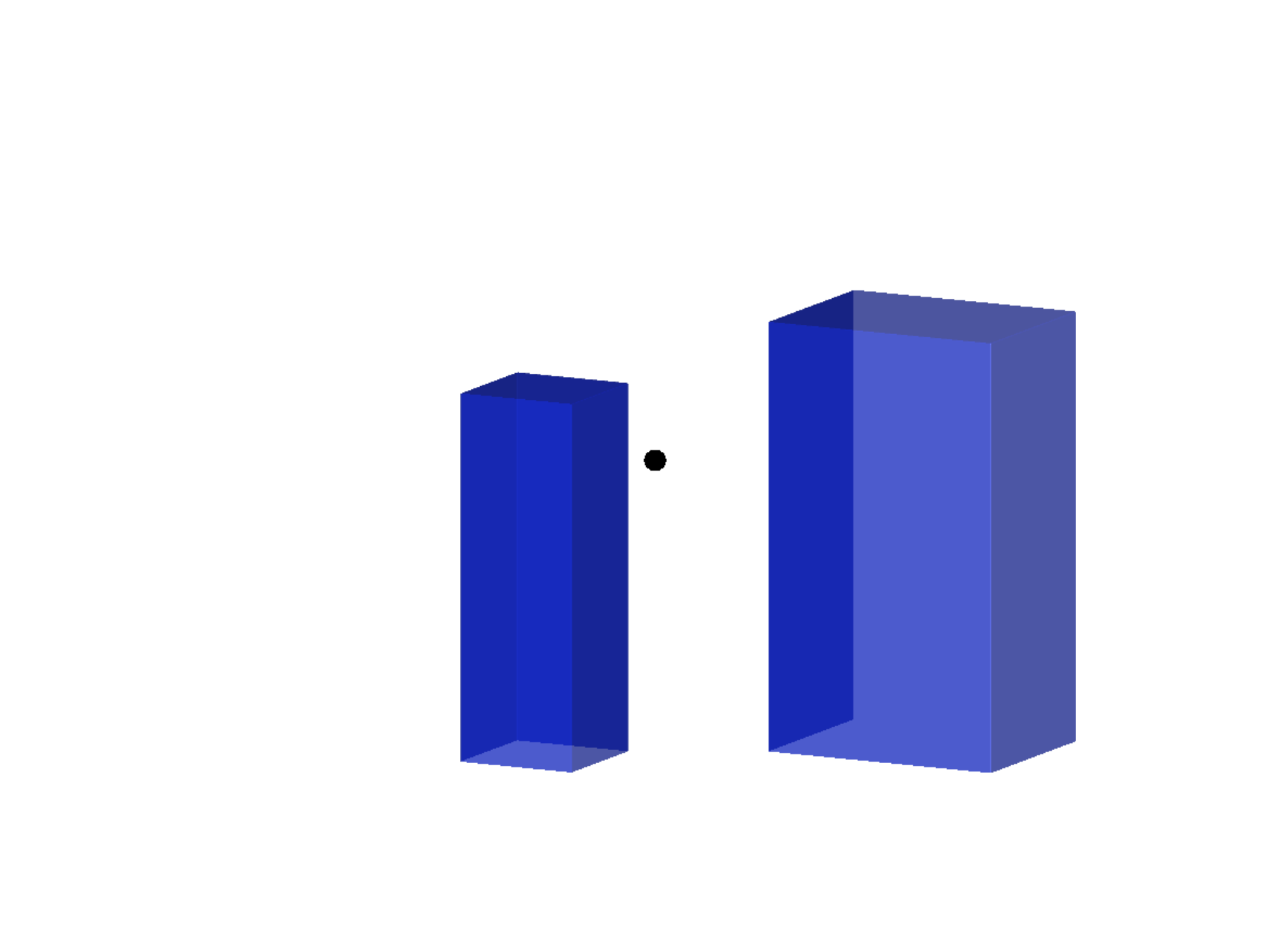}}
\subfigure{\includegraphics[width=\widthtwofigures]{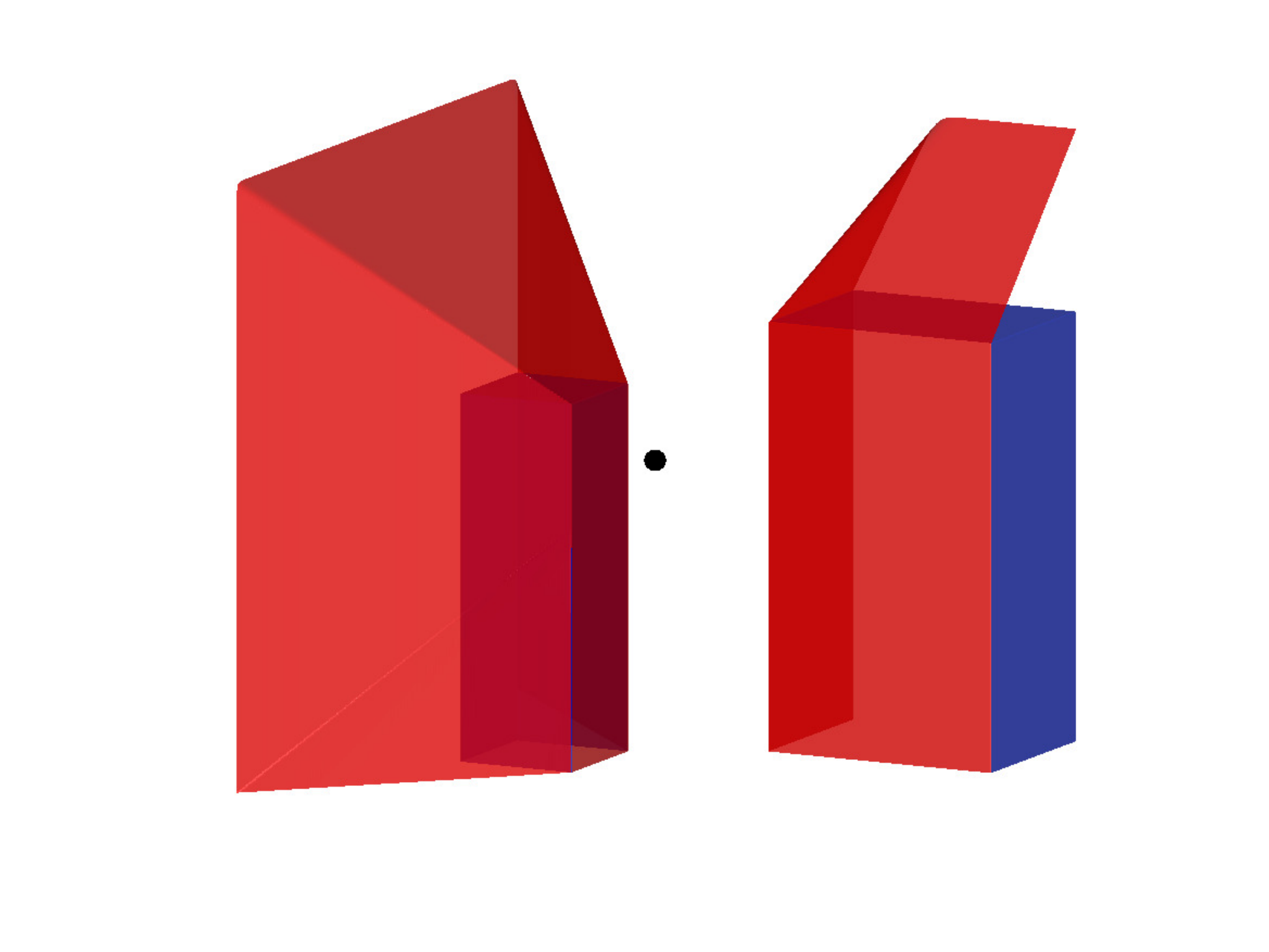}}
\caption{Results for Example \ref{ex3}. The obstacles are displayed in blue, while the contour of the visibility set is displayed in red.}
\label{fig:ex3}
\end{figure}

\begin{example}\label{ex4}
The Stanford 3D Scanning Repository provides a dataset of 35947 distinct points that form what is known as the ``Stanford Bunny''. In this example we considered it as the obstacle by taking the function $g$ to be the signed distance function to the dataset points. The results are displayed in Figure \ref{fig:ex4}.
\end{example}

\begin{figure}[htp]
\centering
\subfigure{\includegraphics[width=\widthtwofigures,trim={2cm 3cm 5cm 1cm},clip]{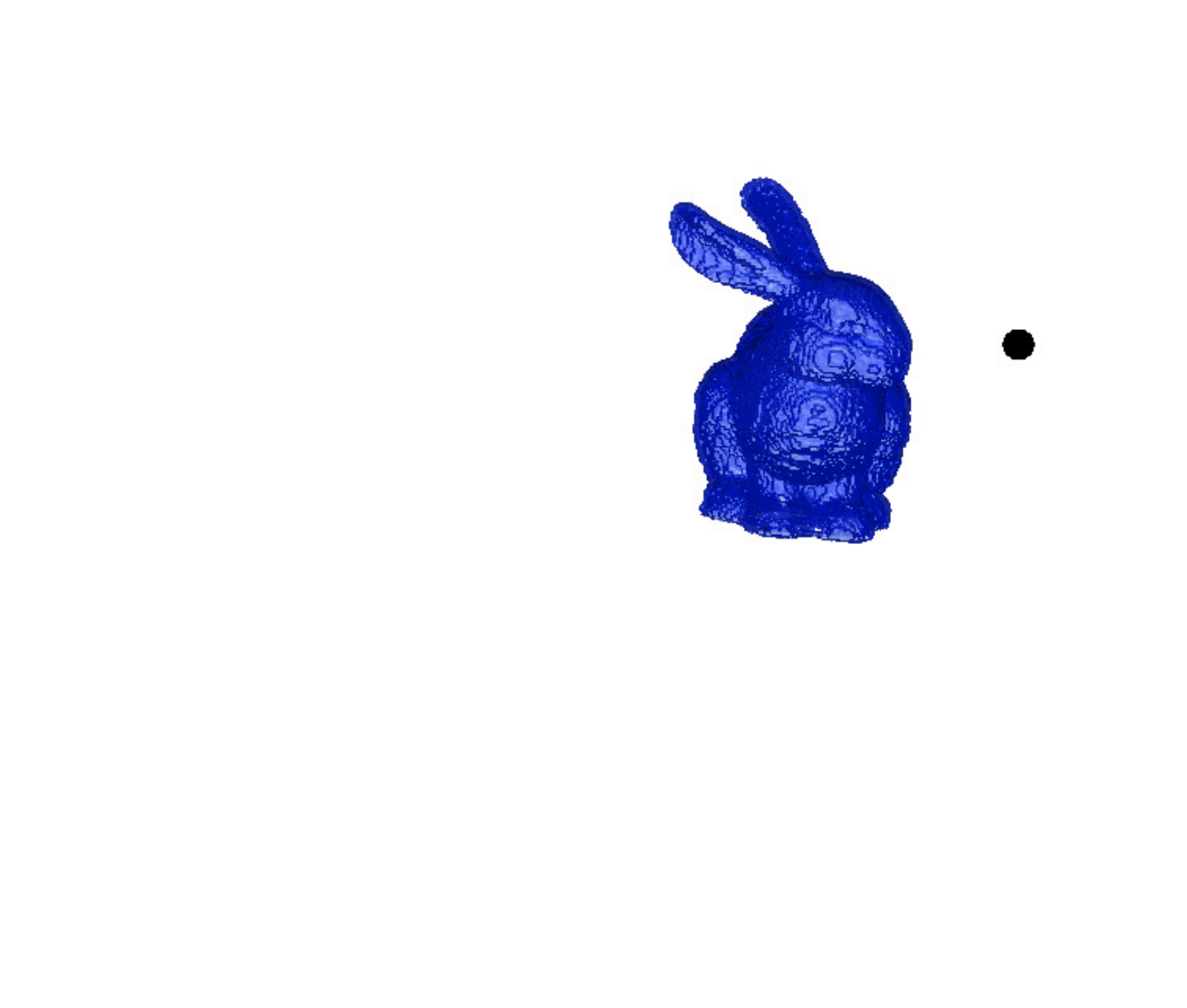}}
\subfigure{\includegraphics[width=\widthtwofigures,trim={2cm 3cm 5cm 1cm},clip]{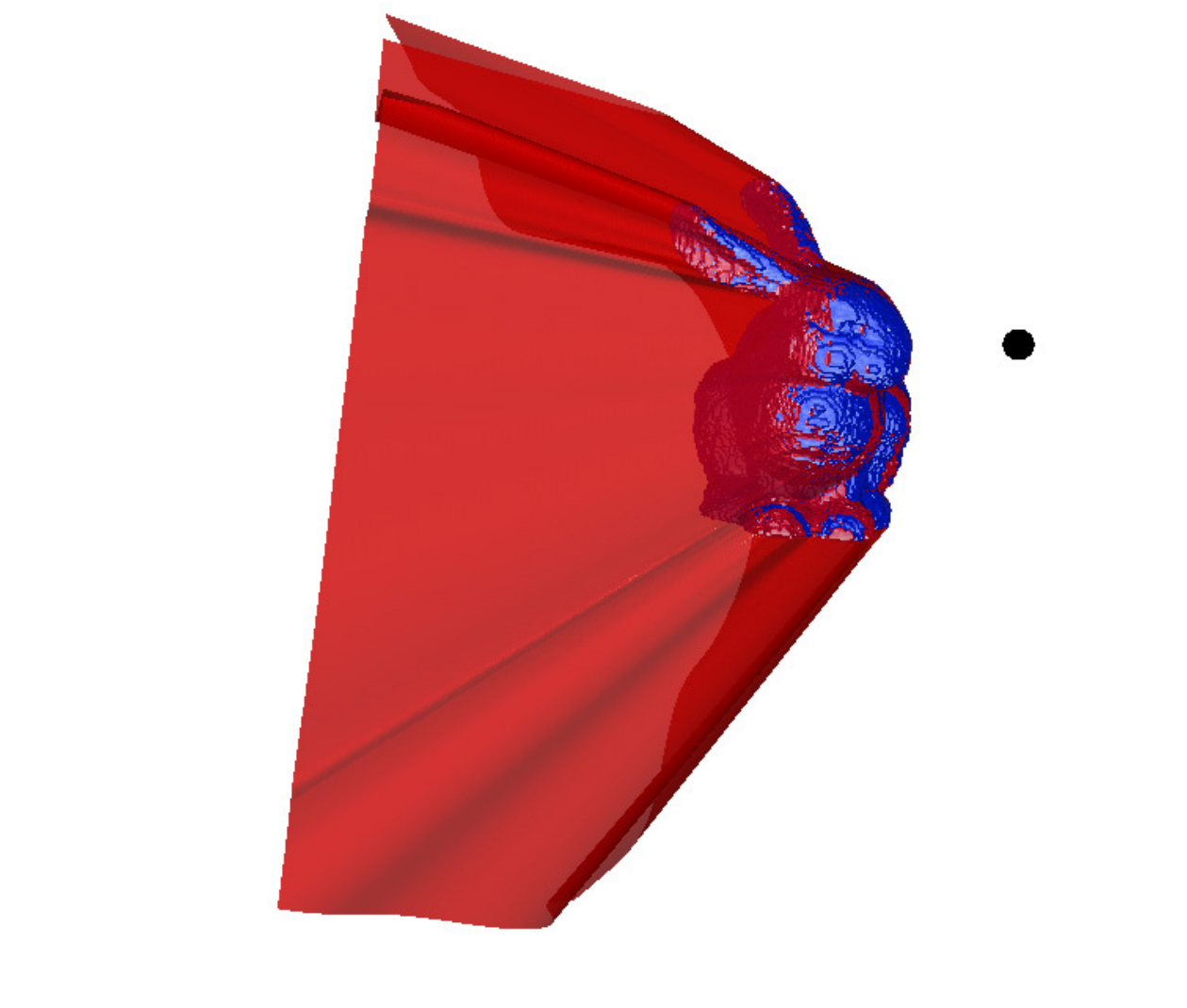}}
\caption{Results for Example \ref{ex4}. The obstacles are displayed in blue, while the contour of the visibility set is displayed in red.}
\label{fig:ex4}
\end{figure}

\section{Conclusions}

In this article, we described a new simpler PDE to compute the visibility set from a given viewpoint given a set of known obstacles. We proposed a finite difference numerical scheme to compute its solution and showed its convergence. We discuss the generalization of the result to multiple viewpoints and present a PDE for the visibility set where a point is visible if it is seen by at least one of many viewpoints. Numerical examples of different visibility sets computed as the solution of the new proposed PDE in both two and three dimensions are presented.

\section*{Acknowledgments}
This material is based upon work supported by the Air Force Office of Scientific Research under award number FA9550-18-1-0167.
The second author thanks the hospitality of the Mathematics and Statistics department of McGill University during its visit where the work for this paper was carried out.

\bibliographystyle{amsalpha}
\bibliography{bibliography}

\end{document}